\documentclass[12pt, a4paper]{article}
\usepackage[margin=1in]{geometry}
\usepackage{amsmath,amssymb,amsthm,bm,mathtools,mathrsfs}
\usepackage{graphicx}
\usepackage{tikz}
\usepackage{hyperref}
\usepackage[authoryear,round]{natbib}
\usepackage{subcaption}
\usepackage{enumitem}
\usepackage{setspace}
\usepackage{xcolor}
\usetikzlibrary{arrows.meta,calc,positioning,decorations.pathreplacing,decorations.markings}
\hypersetup{
    colorlinks=true,
    citecolor=black,
    linkcolor=black,
    urlcolor=black,
    pdfpagemode=UseNone
}

\newtheorem{theorem}{Theorem}[section]
\newtheorem{lemma}[theorem]{Lemma}
\newtheorem{proposition}[theorem]{Proposition}
\newtheorem{corollary}[theorem]{Corollary}
\newtheorem{remark}[theorem]{Remark}

\newtheorem{assumption}{Assumption}
\newtheorem{example}[theorem]{Example}

\newcommand{\R}{\mathbb{R}}
\newcommand{\C}{\mathbb{C}}
\newcommand{\1}{\bm{1}}
\newcommand{\I}{\bm{I}}
\newcommand{\e}{\bm{e}}

\newcommand{\diag}{\operatorname{diag}}
\newcommand{\st}{\, : \,}
\newcommand{\set}[1]{\left\{#1\right\}}
\newcommand{\vect}[1]{\bm{#1}}
\newcommand{\dd}{\mathrm{d}}

\newcommand{\re}{\operatorname{Re}}

\newcommand{\Rop}[2]{\mathcal{R}^{(#1,#2)}}
\newcommand{\Lop}[2]{\mathcal{L}^{(#1,#2)}}
\newcommand{\Prob}{\mathbb{P}}
\newcommand{\E}{\mathbb{E}}

\allowdisplaybreaks

\begin{document}

\title{Colored Markov-Modulated Brownian Motion for Preemptive Workload Stacks}
\author{Maria Laura Battagliola \and Eduardo Medina-Vald\'es \and Oscar Peralta}
\date{}

\maketitle

\begin{abstract}
We develop a colored Markov-modulated Brownian motion model for diffusion-scale systems in which newly created work preempts the active layer, occupies the top of an ordered stack, and must deplete before suspended work resumes. Such preemption arises whenever new work interrupts whatever is currently active, as in last-come, first-served queues and interrupt-driven scheduling, with diffusive dynamics appropriate when the processing rate itself fluctuates randomly. Colors encode stack position and priority, while a background phase, evolving through first-kind transitions that leave the active color unchanged, controls drift and volatility within each color. Second-kind transitions create higher colors at strictly positive random launch heights. A backward recursion summarizes higher-color excursions through return matrices and occupation-density kernels. We obtain stationary product-form representations under hold-and-jump and regulated-base boundary conventions, given explicitly through matrix-analytic formulae, and give matrix-transform formulas for illustrative launch-height families.
\end{abstract}

\section{Introduction}
\label{sec:intro}

Preemptive workload stacks arise when newly arriving or generated work interrupts the currently active item, is placed on top of it, and must be cleared before the suspended work below can resume. Such systems retain memory, since the active item is only the top element of an ordered stack, and each lower item stays suspended until every item above it has been completed. Figure~\ref{fig:intro-stack} illustrates this reduced structure schematically.

In queueing theory, this is the defining stack structure behind last-come, first-served (LCFS) preemptive-resume service, in which a new customer preempts the customer in service, becomes the active job, and the interrupted service later resumes from the point of preemption \citep{HeAlfa1998}. Closely related preemptive-resume priority queues capture the same interruption-and-resumption pattern when arrivals with higher priority displace work with lower priority \citep{Jaiswal1961,Welch1964}; for a classical book treatment of preemptive disciplines, see \citet[Ch.~IV]{Jaiswal1968}. The same reduced structure appears across many non-queueing settings as well. A task in recursive computation may generate a subtask that must be processed before the parent task continues. Solvers in branch-and-bound and tree-search methods repeatedly create, suspend, and resume subproblems \citep{LawlerWood1966}, and work-stealing schedulers redistribute such dynamic computations across workers \citep{BlumofeLeiserson1999}. Likewise, interrupt-driven processors suspend an executing task to service a higher-priority interrupt, nesting arrivals into a stack of suspended execution contexts \citep[Ch.~7]{Stallings2019}. Interrupt handling has also been analyzed as a stack of suspended executions, with corresponding stack-size bounds for interrupt-driven programs \citep{KleimanEykholt1995,ChatterjeeEtAl2004}.

\begin{figure}[htb!]
\centering
\begin{tikzpicture}[>=Latex, thick]

  \node[anchor=east, font=\small] at (-0.15,1.765) {top};
  \node[anchor=east, font=\small] at (-0.15,1.045) {middle};
  \node[anchor=east, font=\small] at (-0.15,0.325) {base};

  \draw[rounded corners, fill=gray!10] (0,0) rectangle (3.0,0.65);
  \draw[rounded corners, fill=gray!22] (0,0.72) rectangle (3.0,1.37);
  \draw[rounded corners, fill=gray!34, draw=blue!55!black, line width=1.1pt] (0,1.44) rectangle (3.0,2.09);
  \node[font=\small] at (1.5,0.325) {suspended};
  \node[font=\small] at (1.5,1.045) {suspended};
  \node[font=\small] at (1.5,1.765) {active};
  \node[font=\small] at (1.5,-0.35) {(a) Before};

  \draw[->] (3.0,1.765) -- (3.7,1.765) -- (3.7,1.045) -- (4.4,1.045);
  \node[font=\small] at (3.35,2.5) {completion};

  \draw[rounded corners, fill=gray!10] (4.4,0) rectangle (7.4,0.65);
  \draw[rounded corners, fill=gray!22, draw=blue!55!black, line width=1.1pt] (4.4,0.72) rectangle (7.4,1.37);
  \node[font=\small] at (5.9,0.325) {suspended};
  \node[font=\small] at (5.9,1.045) {active};
  \node[font=\small] at (5.9,-0.35) {(b) After completion};

  \draw[->, dashed] (7.4,1.045) -- (8.1,1.045) -- (8.1,1.765) -- (8.8,1.765);
  \node[font=\small, align=center] at (8.1,2.5) {new preempting\\work};

  \draw[rounded corners, fill=gray!10] (8.8,0) rectangle (11.8,0.65);
  \draw[rounded corners, fill=gray!22] (8.8,0.72) rectangle (11.8,1.37);
  \draw[rounded corners, fill=gray!34, draw=blue!55!black, line width=1.1pt] (8.8,1.44) rectangle (11.8,2.09);
  \node[font=\small] at (10.3,0.325) {suspended};
  \node[font=\small] at (10.3,1.045) {suspended};
  \node[font=\small] at (10.3,1.765) {active};
  \node[font=\small] at (10.3,-0.35) {(c) After new arrival};

\end{tikzpicture}
\caption{Preemptive workload stack: completion exposes the layer below, whereas new preempting work creates a positive-height active layer on top.}
\label{fig:intro-stack}
\end{figure}
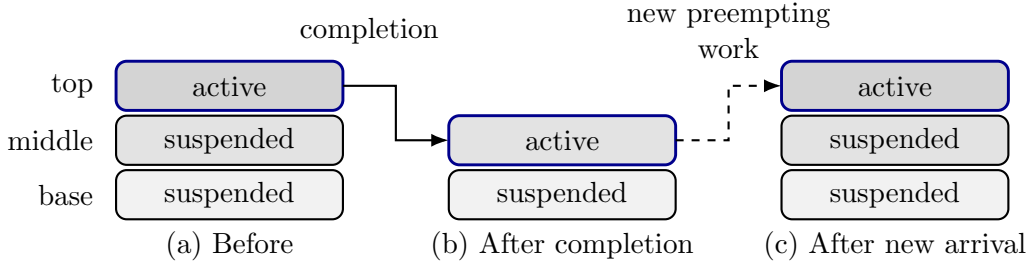

To describe this behavior, \citet{VanHoudt2026} introduced coloring in the setting of Markov-modulated fluid queues under a LCFS preemptive discipline. In particular, a color identifies both the ordered position of a layer in the stack and the priority of the work occupying that position, attributes that are coinciding in this model. The stack holds at most one layer of each color at a time, colors are ordered with increasing indices from base to top, and each layer evolves through first-kind transitions that leave its color unchanged, while a second-kind transition always creates a new layer with a color strictly larger than every color currently present. This framework guarantees that the stack has finite depth and is strictly ordered. Lower coordinates are frozen while a higher color is active, so suspended work is simply stored rather than tracked as a growing history. Then, an excursion into a higher color is analyzed once, and summarized by its return effect on the color below it. Related preemptive and stack structures appear in LCFS preemptive queues with Markovian arrivals and phase-type service \citep{HeAlfa1998} and in stack-and-queue representations of priority queues and tree-like processes \citep{VanHoudtBlondia2006}.

The present paper gives a Brownian, diffusion-scale counterpart of this colored fluid construction. Specifically, Markov-modulated Brownian motion (MMBM) supplies the active-layer dynamics, while coloring supplies the stack memory and the backward recursion, placing the model in the matrix-analytic tradition for MMBMs and fluid queues \citep{Rogers1994,Asmussen1995,AhnRamaswami2003,LatoucheNguyen2015,NguyenPeralta2022,latouche2018analysis}. The proposed model is particularly useful for systems in which the active workload evolves under fluctuating operating conditions, heterogeneous job sizes, and random processing times. In such settings, a Brownian description captures variability around the average net workload rate that a fluid model with fixed deterministic rate cannot represent.

Since in reality preempting work typically arrives as an initial parcel of work, we attach to every transition that creates a new layer a positive random launch height, representing the amount of work carried by the arriving or generated preempting item. This choice also serves a mathematical purpose, as it keeps the newly created layer away from the depletion boundary from the outset, which matters for a Brownian layer. Having started exactly at zero, such a layer would immediately oscillate on both sides of the boundary, leaving it ambiguous which layer is active immediately after a color change. This is unlike the original colored fluid model, where a new layer can be initialized at zero and simply moves away from the boundary at its deterministic rate. With a strictly positive launch height instead, the new layer's later removal is a genuine first-passage event.

Moreover, at the lower boundary we distinguish two operating conventions that reflect different system behaviors. To describe systems that become idle once all admitted work has been completed, we propose the hold-and-jump model, where depletion of the last active layer sends the system to an empty state. The system remains there until a new layer is launched, in the spirit of diffusions with holding and jumping boundary behavior \citep{PengLi2013}. By contrast, the regulated-base model is appropriate for systems with persistent background work, or a lowest-priority workload that remains present at zero. In particular, a base workload remains present and is kept non-negative by a Skorokhod regulator \citep{Skorokhod1961}.

Beyond the description of the proposed model, our main contribution is its stationary analysis under both lower-boundary conventions. Crucially, coloring is not merely a modeling device for distinguishing heterogeneous workload layers. It also provides the computational structure that reduces the joint stationary problem to a finite backward recursion over colors. Complete excursions into higher colors can be summarized through their return effect on the interrupted lower color, leading to explicit product-form stationary densities and a recursive procedure for their evaluation. The distributions of the positive random launch heights enter this recursion through return matrices, allowing a broad range of launch mechanisms to be incorporated within the same framework. The analysis is common to both boundary conventions for higher colors, while the bottom-of-stack calculation is derived separately for the hold-and-jump and regulated-base models. Altogether, this amounts to a fully matrix-analytic treatment of the stationary distribution, with the product-form densities, the return and launch-density kernels, and their normalizing constants all given by explicit matrix formulae rather than by implicit characterizations. Finally, the required return and transfer matrices are obtained from explicit matrix evaluations carried out along a finite backward procedure.

The paper is organized as follows. Section~\ref{sec:background} reviews MMBMs with boundary behavior and the colored fluid framework of \citet{VanHoudt2026}. Section~\ref{sec:model} constructs the colored MMBM workload stack with random launch heights and distinguishes the hold-and-jump and regulated-base variants. Section~\ref{sec:layer_analysis} develops the backward recursion, the occupation-density kernels, and the product-form stationary representations for both variants. Section~\ref{sec:jumps} records the finite computational recursion and evaluates the launch-height transforms for deterministic, phase-type, Weibull, and gamma examples. Section~\ref{sec:conclusion} concludes.

\section{Background}
\label{sec:background}

This background section fixes the preemptive-workload interpretation that motivates the colored stack, recalls the two boundary behaviors of Markov-modulated Brownian motion (MMBM) when the base of the stack reaches zero, and reviews the colored fluid construction of \citet{VanHoudt2026}, which supplies the stack structure and the backward color recursion idea used later.

\subsection{Markov-Modulated Brownian Motion at a Boundary}
\label{sec:background-mmbm}

A Markov-modulated Brownian motion (MMBM) is a diffusive process whose drift and volatility switch according to an underlying continuous-time Markov chain, called the phase process. Formally, let $E$ be a finite phase space, with $|E|=p$, and let $\vect{T}\in\R^{p\times p}$ be the intensity matrix of this continuous-time Markov chain $J=(J(t))_{t\ge0}$ on $E$. Throughout, an intensity matrix is understood to be conservative, with nonnegative off-diagonal entries and zero row sums; we reserve the term subintensity matrix for a matrix with nonnegative off-diagonal entries and nonpositive row sums that need not sum to zero. The diagonal matrices $\vect{\Delta}_{\mu}=\diag(\mu_i\st i\in E)$ and $\vect{\Delta}_{\sigma}=\diag(\sigma_i\st i\in E)$ collect the drift parameters $\mu_i\in\R$ and volatility coefficients $\sigma_i>0$ associated with each phase $i\in E$. This parametrization governs the MMBM under either boundary convention, though the two conventions model the process $(X(t),J(t))$ on $\R_+\times E$ through different stochastic differential equations.

In particular, let $B(t)$ denote a standard Brownian motion. Under the hold-and-jump mechanism, the coordinate solves
\begin{align*}
\dd X(t)=\mu_{J(t)}\,\dd t+\sigma_{J(t)}\,\dd B(t),
\qquad X(0)=x>0,
\end{align*}
up to
\begin{equation}
\label{eq:time_first_passage}
    \tau_0=\inf\{t\ge0\st X(t)=0\},
\end{equation}
its first passage time to $0$. At $\tau_0$, the process enters an absorbing boundary state and remains motionless there until an exogenous relaunch occurs at some later time $\tau_1>\tau_0$. Then, at $\tau_1$, $X$ is reset to a new strictly positive value and the same free equation resumes. This is the form of boundary behavior studied for diffusions with holding and jumping boundaries by \citet{PengLi2013}.

Under the regulated mechanism, the coordinate is not removed at $0$, but rather a regulator pushes it upward when needed to keep it non-negative. A regulated, or reflected, MMBM is the process satisfying
\begin{align*}
\dd X(t)=\mu_{J(t)}\,\dd t+\sigma_{J(t)}\,\dd B(t)+\dd L(t),
\qquad X(0)=x\ge0,
\end{align*}
where $L(t)$ is the minimal Skorokhod regulator. This means that $L$ is the smallest continuous non-decreasing process that keeps $X$ non-negative, and it can increase only when the process is at the boundary. Equivalently, for every $t\ge0$,
\begin{equation}
\label{eq:regulator_properties}
   X(t)\ge 0,\qquad
\int_0^t \1_{\{X(s)>0\}}\,\dd L(s)=0.
\end{equation}
Regulated MMBMs were studied, under the reflected-MMBM terminology, in \citet{Rogers1994,Asmussen1995}.
The remainder of this subsection focuses on the regulated mechanism, which is the standard convention in the MMBM literature, and whose stationary theory is classical and self-contained. The hold-and-jump mechanism, by contrast, requires additional elements to describe, such as the law of the relaunch jumps, so its stationary analysis is deferred to Section~\ref{sec:layer_analysis}, where it is developed directly for the colored workload stack.

A stationary probability row vector of $\vect{T}$ is a non-negative row vector $\vect{\xi}^\top \in \R^p_+$ satisfying $\vect{\xi}\vect{T}=\vect{0}$ and $\vect{\xi}\e=1$, where $\e$ is the all-ones $p$-dimensional column vector, and it is unique when the phase chain is irreducible. For a fixed stationary phase vector $\vect{\xi}$, the usual negative-drift condition for positive recurrence of the regulated process is
\begin{equation}
\label{eq:mmbm-stability}
\vect{\xi}\vect{\Delta}_{\mu}\e<0.
\end{equation}
Under \eqref{eq:mmbm-stability}, the stationary density row vector $\vect{\pi}(x)$, defined by
\begin{align*}
(\vect{\pi}(x))_i\,\dd x=\lim_{t\to\infty}\Prob(X(t)\in \dd x,\ J(t)=i),
\qquad i\in E,\ x>0,
\end{align*}
solves
\begin{align*}
\tfrac12 \vect{\pi}''(x)\vect{\Delta}_{\sigma}^{2}
-\vect{\pi}'(x)\vect{\Delta}_{\mu}
+\vect{\pi}(x)\vect{T}
=\vect{0},
\qquad x>0.
\end{align*}
Seeking decaying solutions of the form $\vect{\pi}(x)=\vect{g}e^{\vect{K}x}$, with $\vect{g}^\top\in \R^p$ and $\vect{K}\in \R^{p \times p}$, leads to the quadratic matrix equation
\begin{equation}
\label{eq:mmbm-U}
\tfrac12 \vect{K}^{2}\vect{\Delta}_{\sigma}^{2}
-\vect{K}\vect{\Delta}_{\mu}
+\vect{T}
=\vect{0}.
\end{equation}
The relevant decaying solution is the one whose eigenvalues lie in the open left half-plane (see \citet{Asmussen1995,AhnRamaswami2017,AhnMeini2020}). Accordingly, the stationary density has the matrix-exponential form
\begin{equation}
\label{eq:mmbm-density}
\vect{\pi}(x)=\vect{g}\,e^{\vect{K}x},
\qquad x>0.
\end{equation}
The missing conditions for $\vect{g}$ come from the reflecting boundary. Imposing the domain condition satisfied by generators of regulated diffusions yields, following \citet{Asmussen1995,AhnMeini2020},
\begin{equation}
\label{eq:mmbm-boundary-g}
\vect{g}\left(\tfrac12\vect{K}\vect{\Delta}_{\sigma}^{2}-\vect{\Delta}_{\mu}\right)=\vect{0}.
\end{equation}
Together with the normalization $\left (\int_0^\infty\vect\pi(x)\,\dd x\right)\,\e=1$, \eqref{eq:mmbm-boundary-g} determines $\vect{g}$.
\subsection{Colored Markov-Modulated Fluid Queues}

\label{sec:background-colored}
\citet{VanHoudt2026} introduced a framework of \emph{colored}
Markov-modulated fluid queues in which fluid is organized into a stack of
layers under a last-come-first-served preemptive discipline. In this subsection, we review this framework.
Let $\mathcal{C} = \{1,\dots,C \} \subset \mathbb{N} $ denote the finite set of color indices in increasing order from base to top. At any time, only one
color is active, namely the highest color currently carrying a positive
fluid level. Specifically, only the active color's level evolves, and the levels held by any lower
colors stay frozen at whatever value they had when a higher color last
became active. Moreover, colors above the active one are empty. Within layer $c \in \mathcal{C}$,
the phase splits into the states $S_-$, at which the coordinate sits at its
lower boundary, and the states $S_+^{(c)}$, at which it is strictly positive.
Within color $c \in \mathcal{C}$, the background transitions are governed by
\begin{align*}
\vect{T}_{\mathrm{VH}}^{(c)}=
\begin{bmatrix}
\vect{T}_{++,\mathrm{VH}}^{(c)} & \vect{T}_{+-,\mathrm{VH}}^{(c)}\\
\vect{T}_{-+,\mathrm{VH}}^{(c)} & \vect{T}_{--,\mathrm{VH}}^{(c)}
\end{bmatrix},
\end{align*}
whereas second-kind transitions with rates
$\vect{T}_{++,\mathrm{VH}}^{(c,\ell)}$ and $\vect{T}_{-+,\mathrm{VH}}^{(c,\ell)}$ create a new layer
$\ell>c$ initialized at zero. Then, working backward from $c=C$ down to $1$, the effect of every excursion into a higher color is folded back into the lower-layer intensity matrices: the rate of leaving the active phase, whether toward the boundary or toward another phase directly, is corrected by adding the rate of triggering a higher-color excursion, weighted by the phase in which control eventually returns once that excursion depletes. Each color's correction therefore requires every color above it to have already been resolved. The stationary distribution then has a
layerwise product-of-matrix-exponentials representation. We refer to
\citet{VanHoudt2026} for the full derivation.
For preemptive workload stacks, this backward recursion is the essential reason for using colors. This avoids keeping track of the full history of suspended jobs, and it only needs the active color, the lower-layer heights, and the recursively computed effect of all possible excursions into higher colors. Our colored MMBM extension preserves this key advantage.

\section{Colored MMBM with Random Launch Heights}
\label{sec:model}

We present the colored MMBM workload stack. This is a diffusive version of the colored fluid model of Section~\ref{sec:background-colored}, as it replaces the deterministic fluid rate of each active color with a MMBM, and launches each newly created color at a random positive height rather than at zero. Throughout, we use a consistent notation with that introduced in Section~\ref{sec:background}. Recall that $\mathcal C=\{1,\ldots,C\}$ denotes the set of workload colors, and let $\mathcal C_0=\{0\}\cup\mathcal C$ denote the enlarged color set used in the hold-and-jump model, where $0$ labels the empty holding state.

We first describe the model heuristically, before giving its precise construction below. To fully describe the process, we need the color $c$, the index recording the position and priority of the active color in the stack, and the descriptors of the MMBM, namely the phase $i\in E$, the state of the background Markov chain governing the drift $\mu_i^{(c)}$, and the volatility $\sigma_i^{(c)}$. We use the common finite phase space $E$ for every color, so the phase process takes values in $E$ throughout regardless of which color is active, and every color's intensity matrix is defined on $E$.
Then, a first-kind transition changes only the phase, leaving the color fixed, while a second-kind transition changes the color by moving to a strictly larger color index, and may simultaneously change the phase as well, moving it to the target phase of the triggering transition. Thus, while a first-kind transition is a phase change alone, a second-kind transition is potentially a joint color-and-phase change. The rates governing both kinds of transitions are made precise below via the subintensity matrices $\vect T^{(c)}$ and the rate matrices $\vect T^{(c,c')}$.

Write $\bm X(t)=(X_1(t),\ldots,X_C(t))^\top$ for the resulting multivariate process, whose $c$-th coordinate is the workload held in color $c$, with the state space and coordinate dynamics made precise below in \eqref{eq:theta-face}–\eqref{eq:ambient-state-spaces} and \eqref{eq:sde-layer-colored}–\eqref{eq:sde-layer-regulated}. There are two equivalent conventions for graphing a colored MMBM. The convention that aligns closely with the queueing literature is to plot the total workload:
\begin{align*}
W(t)=\sum_{k=1}^C X_k(t).
\end{align*}
Our analysis is mainly formulated in the workload-by-color coordinates, where we exploit the multivariate nature of $\bm X(t)$, though we return to the total-workload description at several points below, where aggregate workload admits a simpler interpretation than the individual coordinates. Let $A(t)\in\mathcal C$ denote the active color in the regulated-base model, and let $A(t)\in\mathcal C_0$ in the hold-and-jump model, with $A(t)=0$ while the process is in the empty holding state. When $A(t)=c\geq1$, the workload-by-color graph displays the active coordinate $X_c(t)$ in color $c$. When $A(t)=0$, the stack itself is empty, so the graph correspondingly shows no active workload. On every maximal interval on which $A(t)=c\geq1$, all lower-color coordinates are frozen, and therefore
\begin{align*}
W(t)=\sum_{k<c}X_k(t)+X_c(t),\qquad A(t)=c,
\end{align*}
where $\sum_{k<c}X_k(t)$ is a constant vertical shift on that interval.
Thus, while color $c$ is active, the two graphical representations differ only by the constant shift, and the workload-by-color segment is obtained from the total-workload segment by subtracting the shift. Conversely, the total-workload segment is recovered by adding it back.

The distinction between them becomes important precisely at color changes. Suppose that a second-kind transition launches a higher color $c'>c$ with workload $Y>0$. In the total-workload representation, the new workload is added to the existing stack, so the graph jumps from $W(t-)$ to $W(t-)+Y$. In the workload-by-color representation, each color is measured relative to its own zero level, so the newly active coordinate $X_{c'}(t-)$ jumps from $0$ to $Y$.
When color $c'$ depletes and its path reaches $0$, the most recent nonempty lower color resumes from its frozen height. The total-workload path is continuous at this instant, whereas the workload-by-color graph switches from displaying the path colored in $c'$ to displaying the stored value of the resumed lower coordinate. Any vertical connector used in such a plot therefore represents a change of displayed coordinate rather than a physical jump in total workload.
Note that the scalar path $W(t)$ alone does not determine the full workload vector $\bm X(t)$, since it does not record how workload is distributed among colors; recovering $\bm X(t)$ from $W(t)$ requires appending the entire history of the active color $A(s)$, $s\le t$, which records every color change and hence how each launched workload has been layered into the total.

Moreover, two lower-boundary conventions distinguish what happens once the whole stack reaches its lowest position.
Specifically, after hitting 0, the process then moves to an empty state and waits there until a new color is launched in the hold-and-jump model. In the regulated-base model, the bottom color is never removed. This is encoded in the model by a Skorokhod regulator, which keeps the coordinate non-negative while it continues to evolve as a regulated MMBM. Notice that depletion of any higher color triggers no jump, and control simply returns to the most recent nonempty lower color, resuming continuously from where it was frozen.
Figures~\ref{fig:sample-path-holdjump} and~\ref{fig:sample-path-reg} illustrate the two lower-boundary conventions and, within each figure, the two equivalent graphing conventions.

\begin{figure}[htb!]
\centering
{\bf Hold-and-jump sample paths\par}
\vspace{0.5em}
\begin{minipage}[t]{0.49\textwidth}
\centering
\begin{tikzpicture}[x=0.55cm,y=0.95cm,>=Latex, thick,font=\scriptsize]
  \draw[->] (0,0) -- (8.4,0) node[right] {$t$};
  \draw[->] (0,0) -- (0,3.4) node[above] {$W(t)$};
  \draw[gray!60, densely dotted] (0,0) -- (8.2,0);
  \node[below left] at (0,0) {$0$};

  \draw[very thick,blue!70!black]
    plot[smooth] coordinates {(0.2,0.55) (0.8,0.95) (1.2,1.15) (1.7,1.02) (2.0,1.18)};

  \draw[very thick,red!75!black, dashed, ->] (2.0,1.18) -- (2.0,2.05);
  \node[above left,red!75!black] at (2.0,2.05) {};

  \draw[very thick,red!75!black]
    plot[smooth] coordinates {(2.0,2.05) (2.5,2.45) (3.0,2.20) (3.6,1.75) (4.2,1.40) (4.6,1.18)};

  \draw[very thick,blue!70!black]
    plot[smooth] coordinates {(4.6,1.18) (5.1,0.95) (5.5,0.62) (5.9,0.28) (6.15,0.0)};
  \draw[very thick,blue!70!black] (6.15,0.0) -- (6.85,0.0);
  \node[below,align=center] at (6.5,0.0) {hold at $0$};

  \draw[very thick,blue!70!black, dashed, ->] (7.0,0.0) -- (7.0,0.45);
  \node[above left,blue!70!black] at (7.0,0.45) {};

  \draw[very thick,blue!70!black]
    plot[smooth] coordinates {(7.0,0.45) (7.4,0.78) (7.9,1.00) (8.15,0.88)};

  \node[blue!70!black] at (1.0,1.55) {color $1$};
  \node[red!75!black] at (3.3,2.85) {color $2$};
  \filldraw[black] (4.6,1.18) circle (1.4pt);
\end{tikzpicture}
\subcaption{{Total-workload representation.}}
\label{fig:sample-path-holdjump-total}
\end{minipage}\hfill
\begin{minipage}[t]{0.49\textwidth}
\centering
\begin{tikzpicture}[x=0.55cm,y=0.95cm,>=Latex, thick,font=\scriptsize]
  \draw[->] (0,0) -- (8.4,0) node[right] {$t$};
  \draw[->] (0,0) -- (0,3.4) node[above] {$X_{A(t)}(t)$};

  \node[below left] at (0,0) {$0$};

  \draw[very thick,blue!70!black]
    plot[smooth] coordinates {(0.2,0.55) (0.8,0.95) (1.2,1.15) (1.7,1.02) (2.0,1.18)};

  \draw[gray!75, densely dotted] (2.0,1.18) -- (2.0,0.0);
  \draw[very thick,red!75!black, dashed, ->] (2.0,0.0) -- (2.0,0.87);

  \draw[very thick,red!75!black]
    plot[smooth] coordinates {(2.0,0.87) (2.5,1.27) (3.0,1.02) (3.6,0.57) (4.2,0.22) (4.6,0.0)};

  \draw[gray!75, densely dotted, ->] (4.6,0.0) -- (4.6,1.18);
  \node[blue!70!black,above right] at (4.62,1.18) {resume color $1$};
  \draw[very thick,blue!70!black]
    plot[smooth] coordinates {(4.6,1.18) (5.1,0.95) (5.5,0.62) (5.9,0.28) (6.15,0.0)};
  \draw[very thick,blue!70!black] (6.15,0.0) -- (6.85,0.0);

  \draw[very thick,blue!70!black, dashed, ->] (7.0,0.0) -- (7.0,0.45);
  \draw[very thick,blue!70!black]
    plot[smooth] coordinates {(7.0,0.45) (7.4,0.78) (7.9,1.00) (8.15,0.88)};

  \node[blue!70!black] at (1.0,1.55) {color $1$};
  \node[red!75!black] at (3.2,1.55) {color $2$};
\end{tikzpicture}
\subcaption{{Workload-by-color representation.}}
\label{fig:sample-path-holdjump-color}
\end{minipage}
\caption{{Hold-and-jump paths in total-workload (left panel) and workload-by-color (right panel) coordinates.}}
\label{fig:sample-path-holdjump}
\end{figure}
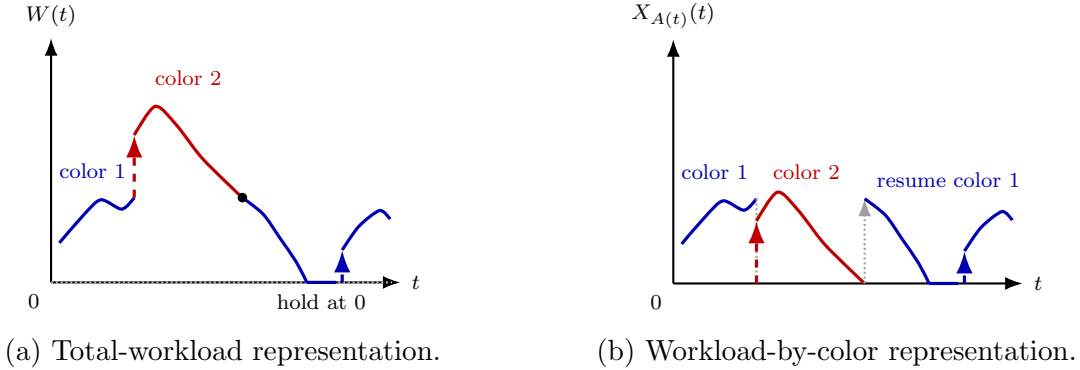

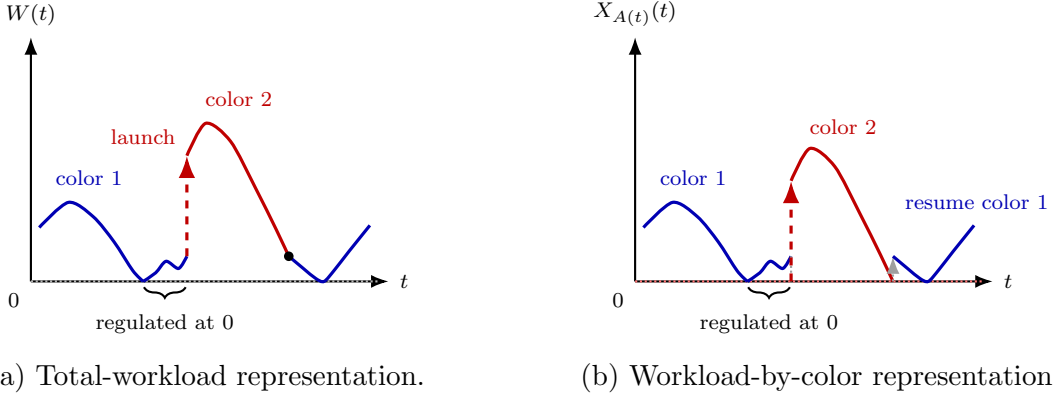
\begin{figure}[htb!]
\centering
{\bf Regulated-base sample paths\par}
\vspace{0.5em}
\begin{minipage}[t]{0.49\textwidth}
\centering
\begin{tikzpicture}[x=0.55cm,y=0.95cm,>=Latex, thick,font=\scriptsize]
  \draw[->] (0,0) -- (8.6,0) node[right] {$t$};
  \draw[->] (0,0) -- (0,3.4) node[above] {$W(t)$};
  \draw[gray!60, densely dotted] (0,0) -- (8.35,0);
  \node[below left] at (0,0) {$0$};

  \draw[very thick,blue!70!black]
    plot[smooth] coordinates {(0.2,0.75) (0.9,1.10) (1.5,0.90) (2.0,0.55) (2.45,0.15) (2.7,0.0)};
  \draw[very thick,blue!70!black]
    plot[smooth] coordinates {(2.7,0.0) (3.0,0.12) (3.25,0.28) (3.55,0.18) (3.75,0.35)};
  \draw[decorate,decoration={brace,mirror,amplitude=4pt}] (2.72,-0.08) -- (3.72,-0.08)
      node[midway,below=6pt,align=center] {regulated at $0$};

  \draw[very thick,red!75!black, dashed, ->] (3.75,0.35) -- (3.75,1.75);
  \node[above left,red!75!black] at (3.75,1.75) {launch};

  \draw[very thick,red!75!black]
    plot[smooth] coordinates {(3.75,1.75) (4.2,2.20) (4.8,1.95) (5.4,1.30) (5.9,0.72) (6.2,0.35)};

  \draw[very thick,blue!70!black]
    plot[smooth] coordinates {(6.2,0.35) (6.55,0.18) (6.85,0.04) (7.05,0.0) (7.3,0.16) (7.65,0.42) (8.15,0.78)};

  \filldraw[black] (6.2,0.35) circle (1.4pt);
  \node[blue!70!black] at (1.4,1.45) {color $1$};
  \node[red!75!black] at (5.0,2.55) {color $2$};
\end{tikzpicture}
\subcaption{{Total-workload representation.}}
\label{fig:sample-path-reg-total}
\end{minipage}\hfill
\begin{minipage}[t]{0.49\textwidth}
\centering
\begin{tikzpicture}[x=0.55cm,y=0.95cm,>=Latex, thick,font=\scriptsize]
  \draw[->] (0,0) -- (8.6,0) node[right] {$t$};
  \draw[->] (0,0) -- (0,3.4) node[above] {$X_{A(t)}(t)$};
  \draw[red!55, densely dotted] (0,0) -- (8.35,0);
  \node[below left] at (0,0) {$0$};

  \draw[very thick,blue!70!black]
    plot[smooth] coordinates {(0.2,0.75) (0.9,1.10) (1.5,0.90) (2.0,0.55) (2.45,0.15) (2.7,0.0)};
  \draw[very thick,blue!70!black]
    plot[smooth] coordinates {(2.7,0.0) (3.0,0.12) (3.25,0.28) (3.55,0.18) (3.75,0.35)};
  \draw[decorate,decoration={brace,mirror,amplitude=4pt}] (2.72,-0.08) -- (3.72,-0.08)
      node[midway,below=6pt,align=center] {regulated at $0$};

  \draw[gray!75, densely dotted] (3.75,0.35) -- (3.75,0.0);
  \draw[very thick,red!75!black, dashed, ->] (3.75,0.0) -- (3.75,1.40);
  \node[red,right] at (3.80,1.05) {};

  \draw[very thick,red!75!black]
    plot[smooth] coordinates {(3.75,1.40) (4.2,1.85) (4.8,1.60) (5.4,0.95) (5.9,0.37) (6.2,0.0)};

  \draw[gray!75, densely dotted, ->] (6.2,0.0) -- (6.2,0.35);
  \node[blue!70!black,above right] at (6.22,0.85) {resume color $1$};
  \draw[very thick,blue!70!black]
    plot[smooth] coordinates {(6.2,0.35) (6.55,0.18) (6.85,0.04) (7.05,0.0) (7.3,0.16) (7.65,0.42) (8.15,0.78)};

  \node[blue!70!black] at (1.4,1.45) {color $1$};
  \node[red!75!black] at (5.0,2.15) {color $2$};
\end{tikzpicture}
\subcaption{{Workload-by-color representation.}}
\label{fig:sample-path-reg-color}
\end{minipage}
\caption{{Regulated-base paths in total-workload (left panel) and workload-by-color (right panel) coordinates. }}
\label{fig:sample-path-reg}
\end{figure}

We now define the state space for the two boundary conventions. For $c\in\mathcal C$, the color-$c$ face is
\begin{equation}
\label{eq:theta-face}
\Theta_c=\set{(x_1,\dots,x_C)\in\R_+^C\st x_c>0,\ x_{c+1}=\cdots=x_C=0},
\end{equation}
consisting of the configurations in which color $c$ is the active color. The coordinates $x_1,\ldots,x_{c-1}$ record whatever work is suspended in the lower colors. The hold-and-jump model requires one further color, labeled $0$, which carries no continuous coordinate of its own, and records that the process is currently holding at the empty boundary. The regulated-base model needs no such color, since its lowest position, color $1$, is regulated rather than depleted, and so is never vacated. The state space is the disjoint union of these faces, paired with the phase space, together with this extra holding face in the hold-and-jump (H) model and a closed base face in the regulated-base (R) model, i.e.
\begin{equation}
\label{eq:ambient-state-spaces}
\mathcal S^{\mathrm H}=\big(\{\bm 0\}\times\{0\}\times E\big)\ \sqcup\ \bigsqcup_{c=1}^C\big(\Theta_c\times\{c\}\times E\big),
\qquad
\mathcal S^{\mathrm R}=\big(\overline\Theta_1\times\{1\}\times E\big)\ \sqcup\ \bigsqcup_{c=2}^C\big(\Theta_c\times\{c\}\times E\big),
\end{equation}
where $\overline\Theta_1=\set{(x_1,0,\ldots,0)\st x_1\ge0}$. Then a state is written as $(\bm x,c,i)$, where $\bm x\in\R_+^C$ is the workload vector, $i\in E$ is the current phase, and $c$ is the active-color label, with $c\in\mathcal C_0$ in the hold-and-jump model and $c\in\mathcal C$ in the regulated-base model.

When an active higher color $c\ge2$ reaches $0$, it is removed and control returns to the highest-indexed color below $c$ that still carries positive workload, if one exists. If no such lower color carries positive workload, the hold-and-jump model enters the empty holding state, whereas the regulated-base model returns to color $1$, which remains active at $0$. In both cases, no positive lower workload remains. In particular, $x_c=0$ is never an occupied state of a depleting color, since the active-color label changes immediately upon depletion. In the regulated-base model, color $1$ is reflected rather than removed at $0$, and may also be frozen there during a higher-color excursion. This is why $x_1=0$ is included in $\overline{\Theta}_1$ in the regulated-base state space.

We can now describe the behavior of the colored MMBM between successive random launch heights, through the joint state $Z(t)=(\bm X(t),A(t),J(t))$, with $\bm X(t)=(X_1(t),\dots, X_C(t))^\top\in\R_+^C$ the current workload vector, $A(t)$ the active color, and $J(t)\in E$ the current phase. Recall that, on the sector $A(t)=c$ for $c\ge2$, only the active coordinate $X_c(t)$ evolves between background jumps, while the lower coordinates remain frozen. Then, its evolution is
\begin{equation}
\label{eq:sde-layer-colored}
\dd X_c(t)=\mu_{J(t)}^{(c)}\,\dd t+\sigma_{J(t)}^{(c)}\,\dd B(t),
\qquad c=2,\dots,C.
\end{equation}
In the hold-and-jump model, color $c=1$ follows this same free dynamics and is depleted at $0$ exactly as the higher colors, after which it is removed. For the base color in the regulated-base model, we instead impose regulation
\begin{equation}
\label{eq:sde-layer-regulated}
\dd X_1(t)=\mu_{J(t)}^{(1)}\,\dd t+\sigma_{J(t)}^{(1)}\,\dd B(t)+\dd L(t),
\qquad X_1(t)\ge0,
\end{equation}
where $L(t)$ is the minimal regulator such that \eqref{eq:regulator_properties} applies.

We assume $\sigma_i^{(c)}>0$ for all $i\in E$ and all $c \in \mathcal{C}$, which keeps every active color genuinely diffusive in every phase.

For each color $c\in \mathcal{C}$, let $\vect{T}^{(c)}\in\R^{p\times p}$ be the subintensity matrix of first-kind phase transitions that leave the active color unchanged, with off-diagonal entries $T^{(c)}_{ij}\ge0$, $i\ne j$, giving first-kind transition rates. For each $1\le c<c'\le C$, let $\vect{T}^{(c,c')}\in\R^{p\times p}$ be a nonnegative matrix whose $(i,j)$-entry gives the rate at which the second-kind transition introduced above carries the active color from $c$ to $c'$ while moving the phase from $i$ to $j$. The diagonal entries of $\vect T^{(c)}$ then account for the total rate of every way phase $i$ can be left, both by a first-kind transition and by a second-kind transition to a higher color, i.e.
\begin{align*}
T^{(c)}_{ii}=-\sum_{j\ne i}T^{(c)}_{ij}-\sum_{c'>c}\sum_{j\in E}T^{(c,c')}_{ij},
\qquad i\in E.
\end{align*}

In the hold-and-jump model, let $\vect{T}^{(0)}\in\R^{p\times p}$ be the analogous subintensity matrix of first-kind phase transitions while the process holds at the empty state, with off-diagonal entries $T^{(0)}_{ij}\ge0$, $i\ne j$, giving first-kind transition rates there. We also allow launches from the empty state by matrices $\vect{T}^{(0,c)}$, $c \in \mathcal{C}$, with $\vect{T}^{(0)}$'s diagonal analogously including the exit rate to every $\vect T^{(0,c)}$. Thus, for each active color $c \in \mathcal{C}$,
\begin{equation}
\label{eq:total-intensity-layer}
\vect{Q}^{(c)}=\vect{T}^{(c)}+\sum_{c'>c}\vect{T}^{(c,c')}
\end{equation}
is the phase intensity matrix while the active color is $c$, with the convention that the sum is empty when $c=C$. In the hold-and-jump model, the empty-state phase process is governed by
\begin{equation}
\label{eq:Q0-holdjump}
\vect{Q}^{(0)}=\vect{T}^{(0)}+\sum_{c=1}^C\vect{T}^{(0,c)}.
\end{equation}
The following Assumption~\ref{assump:intensity} is what \eqref{eq:total-intensity-layer} and \eqref{eq:Q0-holdjump} require in order for $\vect{Q}^{(c)}$ and $\vect{Q}^{(0)}$ to be intensity matrices, and it is assumed throughout the remainder of the paper.
\begin{assumption}[Intensity matrices and irreducibility]
\label{assump:intensity}
For each $c \in \mathcal{C}$,
\begin{align*}
\vect{T}^{(c)}\e+\sum_{c'>c}\vect{T}^{(c,c')}\e=\vect{0},
\end{align*}
and, in the hold-and-jump model,
\begin{align*}
\vect{T}^{(0)}\e+\sum_{c=1}^C\vect{T}^{(0,c)}\e=\vect{0}.
\end{align*}
Moreover $\vect{Q}^{(c)}$ in \eqref{eq:total-intensity-layer} and, in the hold-and-jump model, $\vect{Q}^{(0)}$ in \eqref{eq:Q0-holdjump}, are irreducible on $E$ for every $c \in \mathcal{C}$.
\end{assumption}

We now turn attention to the random launch height mechanism itself. For every admissible pair of colors $0\le c<c'\le C$ appearing in the chosen model and every $i,j\in E$, let $F_{ij}^{(c,c')}$ be a probability distribution on $(0,\infty)$. When a second-kind transition $(c,i)\to(c',j)$ occurs from a state $(\bm{x},c,i)$, where $\bm{x}$ is a realization of $\bm X(t)$, an independent random copy $Y$ of this distribution is sampled, and the new state becomes
\begin{equation}
\label{eq:jump-state-layer}
(\bm{x},c,i)\longmapsto(\bm{x}+Y\bm{e}_{c'},c',j), \qquad Y\sim F_{ij}^{(c,c')},
\end{equation}
where $\bm{e}_{c'}$ is the $C$-dimensional column-vector whose entries are zero apart from the $c'$th coordinate that is 1. Only in the hold-and-jump model do we additionally have boundary launches
\begin{align*}
(\bm 0,0,i)\longmapsto(Y\bm{e}_{c},c,j),
\qquad Y\sim F_{ij}^{(0,c)}.
\end{align*}
In all cases, the random copy $Y$ sampled at a given launch is independent of the Brownian motions driving the active color evolution, of the background phase process, and of every other random copy of $Y$ sampled at every other launch.
With the continuous dynamics, transition intensities, and launch-height mechanism in place, $Z=(Z(t))_{t\ge0}$ is a well-defined non-explosive strong Markov process on $\mathcal S^{\mathrm H}$ or $\mathcal S^{\mathrm R}$.

\begin{remark}
A second-kind transition creates new preempting work, and its random launch height serves two purposes at once. As discussed in the introduction, its strict positivity keeps the newly created Brownian layer away from its own depletion boundary, preventing the immediate downcrossing that a layer launched exactly at zero would otherwise exhibit. At the same time, its value represents the nonzero amount of work such a preempting item carries, for instance its service demand, execution budget, or computational size. Its random size captures variability across such items, and the dependence of the launch-height distribution $F_{ij}^{(c,c')}$ on the source and target colors and phases lets the workload distribution reflect the state of the system at the instant of preemption.
\end{remark}

\section{Backward Recursion and Stationary Distribution}
\label{sec:layer_analysis}

This section constructs the joint stationary law of the colored workload stack in the workload-by-color coordinates $\bm x=(x_1,\ldots,x_C)^\top$, with total workload the derived quantity $w=\sum_{c=1}^C x_c$. The law is described face by face, where a face collects the states in which a fixed subset of colors is occupied and every other color sits at zero. On each face the density factors into a chain of launch-density kernels, one for every occupied color lying above a bottom anchor.

In the hold-and-jump model the anchor is the phase mass at the empty state. In the regulated-base model it is the density on the base face, where color $1$ is active and no higher color is present.

A return matrix records the phase distribution after a complete higher-color excursion and depletion, while a launch-density kernel records the occupation generated while the launched color itself is active. The higher-color recursion is common to both boundary conventions, which differ only in the bottom anchor.

\subsection{Backward recursion for the color exponents}

Fix a depleting color $\ell$. In the hold-and-jump model this means $\ell\in\mathcal C$, whereas in the regulated-base model it means $\ell\in\mathcal C\setminus\{1\}$. Conditional on $A(0)=\ell$ and $X_\ell(0)=y>0$, define the real-time depletion epoch
\begin{align*}
\tau^{(\ell)}_0=\inf\{t\ge0:X_\ell(t)=0\}.
\end{align*}
Let $\vect\Phi_\ell(y)\in\R^{p\times p}$ have entries
\begin{equation}
\label{eq:Phi-y}
(\vect\Phi_\ell(y))_{jk}=\Prob\!\big(J(\tau^{(\ell)}_0)=k \big|\, A(0)=\ell,\ X_\ell(0)=y,\ J(0)=j\big).
\end{equation}
Lemma~\ref{lem:active-censoring} below shows, by backward recursion on the color index, that $\tau_0^{(\ell)}<\infty$ almost surely under the stability condition. Consequently,
\begin{equation}
\label{eq:Phi-rowsum}
\vect\Phi_\ell(y)\e=\e,\qquad y>0,
\end{equation}
so $\vect\Phi_\ell(y)$ is stochastic and its $j$th row is the phase distribution at depletion from launch phase $j$.
For every $0\le c<\ell\le C$ and $i,j\in E$, let $F_{ij}^{(c,\ell)}$ be the jump-height distribution introduced in Section~\ref{sec:model}, and define the averaged depletion kernels
\begin{equation}
\label{eq:avg-Phi}
\overline{\vect{\Phi}}_{ij}^{(c,\ell)}=\int_{(0,\infty)} \bm{e}_{j}^{\top}\vect{\Phi}_\ell(y)\,F_{ij}^{(c,\ell)}(\dd y),
\end{equation}
where $\bm e_j\in\R^p$ denotes the $j$-th standard basis column-vector of the phase space $E$ (not to be confused with the $C$-dimensional $\bm e_{c'}$ of \eqref{eq:jump-state-layer}), and $\overline{\vect{\Phi}}_{ij}^{(c,\ell)}$ is a row vector of dimension $p$. We define the row-wise return operator by
\begin{equation}
\label{eq:return-op-correct}
\Rop{c}{\ell}_{i\bullet}=\sum_{j\in E} T^{(c,\ell)}_{ij}\,\overline{\vect{\Phi}}_{ij}^{(c,\ell)},\qquad i\in E.
\end{equation}
Stacking these row vectors, one for each phase $i\in E$, defines the full matrix $\Rop{c}{\ell}$. Thus, $\Rop{c}{\ell} \in \R^{p\times p}$ converts the second-kind transition rates $\vect T^{(c,\ell)}$ into effective phase-to-phase rates. Specifically, its $(i,k)$-entry is the rate at which a second-kind transition out of phase $i$ triggers a complete excursion into color $\ell$ that eventually returns control to color $c$ in phase $k$.

To define the effective intensity matrix of the process, fix $c\in\mathcal{C}$, and suppose inductively that the higher-color objects have already been constructed for all $\ell>c$. Then, the intensity matrix of the censored color-$c$ dynamics is
\begin{equation}
\label{eq:effective-intensity-correct}
\widetilde{\vect{T}}^{(c)}=\vect{T}^{(c)}+\sum_{\ell>c}\Rop{c}{\ell}.
\end{equation}
Notice that this construction is the diffusion-scale counterpart of the censoring device \citet{VanHoudt2026} uses for colored fluid queues. In the context of censored Markov chains, censoring a process onto a subset of its state space means observing it only during the intervals when it occupies that subset, on a clock that skips every excursion outside it. This censored process is again Markov, with an intensity matrix equal to the original restriction corrected by the net effect of every excursion out of and back into the subset. In \eqref{eq:effective-intensity-correct}, the first-kind subintensity matrix $\vect T^{(c)}$ already describes phase transitions that leave $c$ active while it is active. Then, every remaining way of leaving $c$ is a second-kind transition into some $\ell>c$, which starts a possibly further-nested excursion beyond $\ell$, and eventually returns control to $c$ once $\ell$ depletes. Because the active coordinate $X_c$ does not move while colors with larger indices than $c$ are active, this entire excursion is invisible to color $c$'s own clock, except for its effect on the phase at the moment control returns. This effect is encoded by the return operator $\Rop{c}{\ell}$ of \eqref{eq:return-op-correct}, which replaces the true excursion by a single instantaneous phase jump with the same net rate and return-phase law. Thus, the intensity matrix $\widetilde{\vect T}^{(c)}$ governs all the dynamics, with the phase process observed only while color $c$ is active, with every higher color folded in.
Figure~\ref{fig:censoring} illustrates this idea in terms of total workload. The precise active-time statement is given in Lemma~\ref{lem:active-censoring}.
In the hold-and-jump model, the same censoring idea also defines the empty-state intensity matrix
\begin{equation}
\label{eq:effective-intensity-empty}
\widetilde{\vect T}^{(0)}=\vect T^{(0)}+\sum_{\ell=1}^C\Rop{0}{\ell}.
\end{equation}

\begin{figure}[htb!]
\centering

\begin{subfigure}[t]{0.48\textwidth}
\centering
\resizebox*{!}{3.6cm}{
\begin{tikzpicture}[x=1cm,y=1cm,>=Latex,thick]

\draw[->] (0,0) -- (6.8,0);
\node[above left,font=\footnotesize] at (6.8,0) {$t$};

\draw[->] (0,0) -- (0,2.8);
\draw[gray!60,densely dotted] (0,0) -- (6.6,0);
\node[below left] at (0,0) {$0$};

\draw[very thick,blue!70!black]
  plot[smooth] coordinates {
    (0.2,0.55)
    (0.8,0.95)
    (1.2,1.15)
    (1.7,1.02)
    (2.0,1.18)
  };

\draw[very thick,red!75!black,dashed,->]
  (2.0,1.18) -- (2.0,2.05);
\node[above left,red!75!black,font=\footnotesize]
  at (2.0,2.05) {launch};

\draw[very thick,red!75!black]
  plot[smooth] coordinates {
    (2.0,2.05)
    (2.6,2.4)
    (3.2,2.1)
    (3.8,1.7)
    (4.6,1.18)
  };

\draw[very thick,blue!70!black]
  plot[smooth] coordinates {
    (4.6,1.18)
    (5.0,0.95)
    (5.4,0.75)
    (5.8,0.85)
    (6.2,0.65)
  };

\node[blue!70!black,font=\footnotesize]
  at (1.1,1.55) {color $c$};
\node[red!75!black,font=\footnotesize]
  at (3.2,2.6) {color $\ell$};
\node[blue!70!black,font=\footnotesize]
  at (5.7,1.15) {color $c$};

\draw[densely dashed,gray]
  (2.0,0) -- (2.0,1.18);
\draw[densely dashed,gray]
  (4.6,0) -- (4.6,1.18);

\node[below] at (2.0,0) {$t_1$};
\node[below] at (4.6,0) {$t_2$};

\draw[<->,gray]
  (2.08,0.28) -- (4.52,0.28)
  node[midway,above,font=\footnotesize]
  {folded interval};

\filldraw[black] (4.6,1.18) circle (1.3pt);

\end{tikzpicture}
}
\caption{Real time}
\label{fig:censoring-a}
\end{subfigure}
\hfill
\begin{subfigure}[t]{0.48\textwidth}
\centering
\resizebox*{!}{3.6cm}{
\begin{tikzpicture}[x=1.35cm,y=1cm,>=Latex,thick]

\draw[->] (0,0) -- (4.2,0);
\node[above left,font=\footnotesize]
  at (4.2,0) {censored time};

\draw[->] (0,0) -- (0,2.8);
\draw[gray!60,densely dotted] (0,0) -- (4.0,0);
\node[below left] at (0,0) {$0$};

\draw[very thick,blue!70!black]
  plot[smooth] coordinates {
    (0.2,0.55)
    (0.8,0.95)
    (1.2,1.15)
    (1.7,1.02)
    (2.0,1.18)
  };

\draw[very thick,blue!70!black]
  plot[smooth] coordinates {
    (2.0,1.18)
    (2.4,0.95)
    (2.8,0.75)
    (3.2,0.85)
    (3.6,0.65)
  };

\node[blue!70!black,font=\footnotesize]
  at (1.1,1.55) {color $c$};

\draw[densely dashed,gray]
  (2.0,0) -- (2.0,2.0);
\filldraw[black] (2.0,1.18) circle (1.3pt);

\node[below,font=\footnotesize,align=center]
  at (2.0,0)
  {$t_1$ and $t_2$ identified};

\node[font=\footnotesize,align=center]
  at (2.0,2.4)
  {instantaneous phase jump\\
   governed by $\Rop{c}{\ell}$};

\end{tikzpicture}
}
\caption{Censored clock for color $c$}
\label{fig:censoring-b}
\end{subfigure}

\caption{Active-time censoring for color $c$ in the total-workload
framework. In real time, the higher-color excursion occupies the interval
$(t_1,t_2)$ (left panel). Under the censored clock for color $c$, this
interval is deleted, so that $t_1$ and $t_2$ are folded into the same
instant and the excursion is replaced by an instantaneous phase transition
governed by $\Rop{c}{\ell}$ (right panel).}
\label{fig:censoring}
\end{figure}
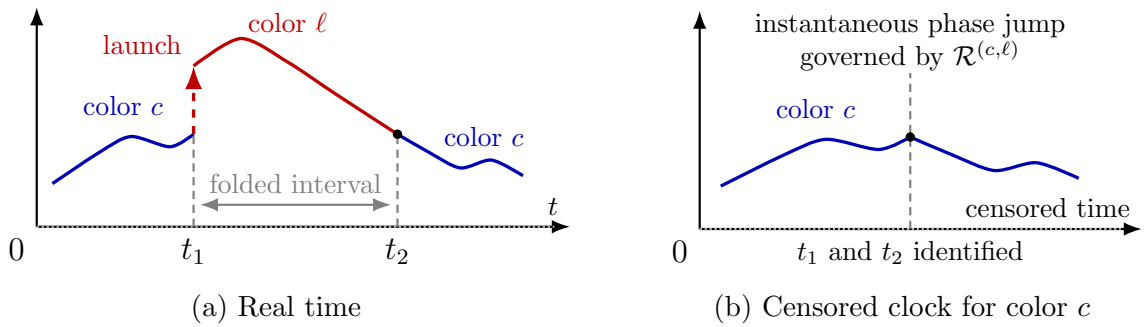

As in Section~\ref{sec:background-mmbm}, for each color $c \in \mathcal{C}$ we write
\begin{align*}
\vect{\Delta}_{\mu}^{(c)}=\diag(\mu_i^{(c)}\st i\in E),
\qquad
\vect{\Delta}_{\sigma}^{(c)}=\diag(\sigma_i^{(c)}\st i\in E),
\end{align*}
for the diagonal drift and volatility matrices governing color $c$.
{
\begin{assumption}[Recursive per-color stability]
\label{assump:stability}
For every color $c\in\mathcal C$, taken in decreasing order from $c=C$, the effective intensity matrix $\widetilde{\vect T}^{(c)}$ formed from the return operators of all colors above $c$ is irreducible and, with invariant probability row vector $\vect\xi^{(c)}$, satisfies
\begin{equation}
\label{eq:neg_drift}
\vect\xi^{(c)}\vect\Delta_\mu^{(c)}\e<0.
\end{equation}
In the hold-and-jump model, the effective empty-state intensity matrix $\widetilde{\vect T}^{(0)}$ in \eqref{eq:effective-intensity-empty} is also irreducible.
\end{assumption}

This is a condition on the recursively censored dynamics rather than on the primitive matrices $\vect T^{(c)}$ and $\vect T^{(c,\ell)}$ alone, so it is verified along the backward recursion itself, one color at a time.
}

\begin{lemma}
\label{lem:active-censoring}
Suppose that Assumptions~\ref{assump:intensity} and~\ref{assump:stability} hold, and let $\ell$ be a depleting color whose higher-color return operators have already been constructed. Define its accumulated active time and right-continuous inverse by
\begin{align*}
\kappa_\ell(t)&=\int_0^t\1_{\{A(s)=\ell\}}\,\dd s,
&
\gamma_\ell(u)&=\inf\{t\ge0:\kappa_\ell(t)>u\}.
\end{align*}
Up to depletion, the time-changed process
\begin{align*}
\widehat X_\ell(u)=X_\ell(\gamma_\ell(u)),
\qquad
\widehat J_\ell(u)=J(\gamma_\ell(u))
\end{align*}
is an MMBM with drift and volatility matrices $\vect\Delta_\mu^{(\ell)}$ and $\vect\Delta_\sigma^{(\ell)}$, and phase intensity matrix $\widetilde{\vect T}^{(\ell)}$. Its depletion time $\widehat\tau_0^{(\ell)}=\inf\{u\ge0:\widehat X_\ell(u)=0\}$ is almost surely finite, the corresponding real-time depletion epoch $\tau_0^{(\ell)}$ is almost surely finite, and $J(\tau_0^{(\ell)})=\widehat J_\ell(\widehat\tau_0^{(\ell)})$ almost surely. Moreover, for every $y>0$, $x>0$, and $j,k\in E$,
\begin{align*}
\E&\left[\int_0^{\tau_0^{(\ell)}}\1_{\{A(s)=\ell,\,J(s)=k,\,X_\ell(s)\in\dd x\}}\,\dd s\ \Big|\ A(0)=\ell,\,X_\ell(0)=y,\,J(0)=j\right]
\\
&=
\E\left[\int_0^{\widehat\tau_0^{(\ell)}}\1_{\{\widehat J_\ell(u)=k,\,\widehat X_\ell(u)\in\dd x\}}\,\dd u\ \Big|\ \widehat X_\ell(0)=y,\,\widehat J_\ell(0)=j\right],
\end{align*}
so the two occupation densities agree entry by entry as $p\times p$ matrices.
\end{lemma}

\begin{proof}
Censoring deletes every interval during which a color above $\ell$ is active. At a second-kind transition from phase $i$ into a higher color $m$, the strong Markov property replaces the deleted excursion by an instantaneous return to phase $k$ with rate given by the $(i,k)$ entry of $\Rop{\ell}{m}$. Hence the censored phase intensity is $\vect T^{(\ell)}+\sum_{m>\ell}\Rop{\ell}{m}=\widetilde{\vect T}^{(\ell)}$, while the coordinate coefficients are unchanged because $X_\ell$ is frozen throughout every deleted excursion. The negative censored drift in \eqref{eq:neg_drift} implies almost-sure depletion on the active-time clock. Real-time finiteness follows by backward induction: it is immediate for $\ell=C$, and for a lower color only finitely many higher-color launches can occur before its finite active-time depletion, because launch intensities are bounded, while every launched higher excursion has finite real duration by the induction hypothesis. The matrix identity follows directly from the definition of the inverse clock, since summing the real-time occupation over the periods with $A(s)=\ell$ reproduces the occupation on the active-time clock, and the depletion phase is the same on both clocks.
\end{proof}

{
For the top color $C$, $\widetilde{\vect T}^{(C)}=\vect T^{(C)}$. Lemma~\ref{lem:active-censoring} gives $\overline{\vect\Phi}_{ij}^{(c,\ell)}\e=1$, and therefore $\Rop{c}{\ell}\e=\vect T^{(c,\ell)}\e$. Consequently, at each backward step,
\begin{equation}
\label{eq:Ttilde-rowsum}
\widetilde{\vect T}^{(c)}\e
=\vect T^{(c)}\e+\sum_{\ell>c}\vect T^{(c,\ell)}\e
=\vect Q^{(c)}\e
=\vect0.
\end{equation}
Thus every recursively constructed $\widetilde{\vect T}^{(c)}$ is conservative. The same argument applies to $\widetilde{\vect T}^{(0)}$ in the hold-and-jump model.
}

The following result shows that, under the stated regularity conditions, the depletion probability matrix admits a matrix-exponential representation.

\begin{proposition}
\label{prop:backward}
{Suppose that Assumptions~\ref{assump:intensity} and~\ref{assump:stability} hold.}
Then, using the active-time MMBM of Lemma~\ref{lem:active-censoring} and iterating from $c=C$ down to $c=2$, the depletion probability matrix defined in \eqref{eq:Phi-y} admits the representation
\begin{align*}
\vect{\Phi}_c(y)=e^{\vect{U}_c y},\qquad y>0,
\end{align*}
where $\vect U_c$ is the unique depletion exponent solving
\begin{equation}
\label{eq:layer-qme}
\tfrac12(\vect{\Delta}_{\sigma}^{(c)})^2\vect U_c^2
+\vect{\Delta}_{\mu}^{(c)}\vect U_c
+\widetilde{\vect T}^{(c)}
=\vect0,
\end{equation}
for which $0$ is a simple eigenvalue with right eigenvector $\e$, and all nonzero eigenvalues lie in the open left half-plane.

{At color $1$, under the corresponding conditions in Assumption~\ref{assump:stability}:}
\begin{itemize}
\item \emph{Hold-and-jump model.} Color $1$ is depleting, and \eqref{eq:layer-qme} and {\eqref{eq:neg_drift}} again give a unique depletion exponent $\vect U_1$, with $\vect\Phi_1(y)=e^{\vect U_1 y}$, $y>0$.

\item \emph{Regulated-base model.} Color $1$ is regulated rather than depleted, and there is instead a unique stationary regulated-MMBM decay exponent $\vect K_{\mathrm{reg}}$ solving
\begin{equation}
\label{eq:regulated-base-qme}
\tfrac12 \vect{K}_{\mathrm{reg}}^{2}(\vect{\Delta}_{\sigma}^{(1)})^{2}-\vect{K}_{\mathrm{reg}}\vect{\Delta}_{\mu}^{(1)}+\widetilde{\vect{T}}^{(1)}=\vect{0},
\end{equation}
with eigenvalues in the open left half-plane, giving the anchor density
\begin{equation}
\label{eq:regulated-base-density}
\vect{\nu}^{\mathrm R}(x)=\vect{g}_{\mathrm{reg}}e^{\vect{K}_{\mathrm{reg}}x},
\qquad x>0,
\end{equation}
where $\vect{g}_{\mathrm{reg}}$ satisfies the reflecting boundary condition
\begin{equation}
\label{eq:regulated-base-boundary-g}
\vect{g}_{\mathrm{reg}}\left(\tfrac12\vect{K}_{\mathrm{reg}}(\vect{\Delta}_{\sigma}^{(1)})^{2}-\vect{\Delta}_{\mu}^{(1)}\right)=\vect{0},
\end{equation}
fixing its direction up to scale.
\end{itemize}
\end{proposition}

\begin{proof}
By Lemma~\ref{lem:active-censoring}, the process observed while a depleting color is active is the MMBM with phase intensity $\widetilde{\vect T}^{(c)}$, and its phase at active-time depletion coincides with its phase at real-time depletion. The proof is therefore a backward induction using the standard {first-passage} quadratic matrix equation on the phase space $E$. For $c=C$, \eqref{eq:layer-qme} is precisely the classical {first-passage depletion} equation with phase intensity matrix $\vect{T}^{(C)}$, and {\eqref{eq:neg_drift}} yields a unique depletion exponent {$\vect{U}_C$}, with the zero eigenvalue (right eigenvector $\e$) forced by the zero row sums of the intensity matrix $\vect T^{(C)}$ and the remaining eigenvalues confined to the open left half-plane by the negative-drift condition, this eigenvalue placement being the classical result for first-passage quadratic matrix equations of this type \citep{Rogers1994,Asmussen1995,AhnRamaswami2017,AhnMeini2020}. Assume next that the color objects have been constructed for all $\ell>c$. By \eqref{eq:Ttilde-rowsum}, $\widetilde{\vect{T}}^{(c)}$ is an intensity matrix, and {\eqref{eq:neg_drift}} is the usual negative-drift condition for the MMBM. Applying the same existence-and-uniqueness result yields the unique depletion exponent {$\vect{U}_c$} of \eqref{eq:layer-qme}. In the hold-and-jump model, the induction simply continues one step further to $c=1$, since color $1$ is depleting there and $\widetilde{\vect T}^{(1)}$ satisfies the same hypotheses. In the regulated-base model, color $1$ is never depleted, so it falls outside this induction entirely; instead, $\vect K_{\mathrm{reg}}$ follows from the classical spectral theory for a regulated MMBM with intensity matrix $\widetilde{\vect T}^{(1)}$, recalled in Section~\ref{sec:background-mmbm} and in \eqref{eq:mmbm-U}--\eqref{eq:mmbm-density} \citep{Asmussen1995,Rogers1994,AhnRamaswami2017,AhnMeini2020}, applied in place of $\vect T$.
\end{proof}
\begin{remark}
The function $\vect\nu^{\mathrm R}(x_1)$ is the phase-resolved density on $\overline\Theta_1\times\{1\}\times E$, the sector of the state space introduced in \eqref{eq:ambient-state-spaces} in which color $1$ is active and every higher color is empty. It is not, in general, the marginal density of $X_1$ obtained after also including the mass held in states where some higher color is active. Its scale, together with that mass, is determined by the global normalization in Theorem~\ref{thm:main-reg}.
\end{remark}

From a numerical perspective, the depletion exponents are obtained recursively, starting from the highest color and solving the classical first-passage MMBM quadratic matrix equation \eqref{eq:layer-qme} at each step. This recursion extends to $c=1$ in the hold-and-jump model, whereas the regulated-base model instead requires the solution of \eqref{eq:regulated-base-qme} together with the boundary equation \eqref{eq:regulated-base-boundary-g}. Efficient and reliable numerical methods for these matrix equations are available \citep{AhnRamaswami2017,AhnMeini2020}.

\subsection{Matrix transforms and occupation densities}
{
With the depletion exponent $\vect U_\ell$ constructed in Proposition~\ref{prop:backward}, fix $0\le c<\ell\le C$ and define the scalar transform
\begin{equation}
\label{eq:scalar-transform}
\mathscr M_{ij}^{(c,\ell)}(s)
=\int_{(0,\infty)}e^{sy}F_{ij}^{(c,\ell)}(\dd y),
\qquad \operatorname{Re}(s)\le0.
\end{equation}
For every probability law $F_{ij}^{(c,\ell)}$, this transform is finite on the closed left half-plane. Its matrix evaluation at the depletion exponent is defined directly by
\begin{equation}
\label{eq:matrix-transform}
\mathscr M_{ij}^{(c,\ell)}(\vect U_\ell)
:=\int_{(0,\infty)}e^{\vect U_\ell y}F_{ij}^{(c,\ell)}(\dd y).
\end{equation}
Because $e^{\vect U_\ell y}=\vect\Phi_\ell(y)$ is stochastic, the integral in \eqref{eq:matrix-transform} is well defined for every launch-height distribution, without a moment condition. Moreover,
\begin{align*}
\overline{\vect\Phi}_{ij}^{(c,\ell)}
=\bm e_j^\top\mathscr M_{ij}^{(c,\ell)}(\vect U_\ell),
\end{align*}
which recovers \eqref{eq:avg-Phi} and gives the return operator \eqref{eq:return-op-correct} in closed form.

\begin{assumption}[Finite launch-height means]
\label{assump:first-moment}
For every relevant pair of colors $(c,\ell)$ and every $i,j\in E$, the launch height $H_{ij}^{(c,\ell)}\sim F_{ij}^{(c,\ell)}$ satisfies
\begin{align*}
\E[H_{ij}^{(c,\ell)}]<\infty.
\end{align*}
\end{assumption}
The return matrices require no moment condition. The first moment enters only through the occupation kernel, whose entries grow linearly in the launch height, so Assumption~\ref{assump:first-moment} is what makes the integrated transfer matrices and the stationary normalization finite.
}

\begin{remark}
\label{rem:functional-calculus}
{
Suppose in addition that there exists $\varepsilon_{ij}>0$ with $\E[e^{\varepsilon_{ij}H_{ij}^{(c,\ell)}}]<\infty$. Then $\mathscr M_{ij}^{(c,\ell)}$ is analytic on $\{z\in\C:\re(z)<\varepsilon_{ij}\}$, an open neighborhood of $\sigma(\vect U_\ell)$, and the holomorphic functional calculus gives
\begin{equation}
\label{eq:matrix-lt}
\mathscr M_{ij}^{(c,\ell)}(\vect U_\ell)
=\frac{1}{2\pi\mathrm i}\oint_\Gamma
\mathscr M_{ij}^{(c,\ell)}(z)
(z\I_p-\vect U_\ell)^{-1}\,\dd z
=\int_{(0,\infty)}e^{\vect U_\ell y}F_{ij}^{(c,\ell)}(\dd y),
\end{equation}
where $\Gamma$ is any positively oriented contour contained in that domain and enclosing $\sigma(\vect U_\ell)$. The exponential-moment condition is therefore what the contour representation requires, whereas neither the matrix transform nor the stationary product form depends on it. See \citet{bladt2016use} for further details.
}
\end{remark}

{The transform $\mathscr M_{ij}^{(c,\ell)}(\vect U_\ell)$ settles the return calculation. The stationary densities require in addition the occupation accumulated by a color during its own active periods.}
Let $\vect{G}_\ell(y,x)$ denote the killed-at-zero Green kernel, defined by
\begin{equation}
\label{eq:green-def}
{(\vect G_\ell(y,x))_{jk}\,\dd x
=\E\left[\int_0^{\tau_0^{(\ell)}}
\1_{\{A(s)=\ell,\ J(s)=k,\ X_\ell(s)\in\dd x\}}\,\dd s
\ \Big|\ A(0)=\ell,\ X_\ell(0)=y,\ J(0)=j\right],}
\end{equation}
{the expected real time spent in phase $k$ with color $\ell$ active and its level in $\dd x$ before depletion, starting from active color $\ell$, level $y$, and phase $j$.} The launch-density kernel is
\begin{equation}
\label{eq:launch-op-correct}
\Lop{c}{\ell}(x)_{i\bullet}=\sum_{j\in E}T^{(c,\ell)}_{ij}\int_{(0,\infty)}\bm{e}_{j}^{\top}\vect{G}_\ell(y,x)\,F_{ij}^{(c,\ell)}(\dd y),\qquad i\in E.
\end{equation}
Stacking these row vectors, one for each phase $i\in E$, defines the full matrix $\Lop{c}{\ell}(x)$. In words, $\Lop{c}{\ell}(x)$ describes what happens after a single second-kind transition creates a new excursion into color $\ell$: starting from phase $i$ while color $c$ is active, a transition to phase $j$ at rate $T^{(c,\ell)}_{ij}$ launches color $\ell$ at a random height $Y\sim F_{ij}^{(c,\ell)}$, and the Green kernel $\vect G_\ell(Y,x)$ then gives the expected time {for which color $\ell$ is active at level $x$}, across every phase, before it depletes and returns control back down. Averaging over the random launch height and summing over the target phase $j$, $\Lop{c}{\ell}(x)$ is therefore the expected occupation-time density, at level $x$, {during the periods when color $\ell$ is active} following a launch out of phase $i$ in color $c$.

The following result, whose proof can be found in Appendix~\ref{app:layer-derivations}, states that, under previously established regularity conditions, the Green kernel in \eqref{eq:green-def} takes a closed formula in our model.
\begin{lemma}
\label{lem:green}
{Suppose that Assumptions~\ref{assump:intensity} and~\ref{assump:stability} hold, and fix a depleting color $\ell$.} Then $\vect U_\ell$ solves
\eqref{eq:layer-qme}, has a simple zero eigenvalue with right
eigenvector $\e$, and has all remaining eigenvalues in the open left
half-plane.
Let $\vect V_\ell$ denote the complementary solution of
\begin{equation}
\label{eq:V-def}
\tfrac12(\vect\Delta_\sigma^{(\ell)})^2\vect V_\ell^2
+\vect\Delta_\mu^{(\ell)}\vect V_\ell
+\widetilde{\vect T}^{(\ell)}
=\vect 0,
\end{equation}
whose eigenvalues lie in the open right half-plane. Standard MMBM
Wiener--Hopf theory gives this complementary solution and the
nonsingularity of $\vect V_\ell-\vect U_\ell$
\citep{Asmussen1995,AhnMeini2020}. Define
\begin{equation}
\label{eq:W-def}
\vect W_\ell
:=2\big[(\vect\Delta_\sigma^{(\ell)})^2
(\vect V_\ell-\vect U_\ell)\big]^{-1}.
\end{equation}
Then, the Green kernel appearing in \eqref{eq:launch-op-correct} is
\begin{equation}
\label{eq:green-closed-form}
\vect G_\ell(y,x)=
\begin{cases}
\big(e^{\vect V_\ell y}-e^{\vect U_\ell y}\big)
e^{-\vect V_\ell x}\,\vect W_\ell,
& 0\le y\le x,\\[4pt]
\big(e^{\vect U_\ell(y-x)}
-e^{\vect U_\ell y}e^{-\vect V_\ell x}\big)\vect W_\ell,
& y\ge x>0.
\end{cases}
\end{equation}
\end{lemma}

Integrating the Green kernel over the active level gives the total phase-resolved occupation time before depletion.
\begin{lemma}
\label{lem:occupation}
Suppose that Assumptions~\ref{assump:intensity} and~\ref{assump:stability} hold, and fix a depleting color $\ell\in\mathcal C$. Let
$\vect m_\ell(y)=\int_0^\infty\vect G_\ell(y,x)\,\dd x$. Then
\begin{equation}
\label{eq:occupation-closed-form}
\vect m_\ell(y)=y\vect P_\ell+(\I_p-e^{\vect U_\ell y})\vect Q_\ell,
\end{equation}
where $\vect P_\ell=|d_\ell|^{-1}\e\vect\xi^{(\ell)}$,
$d_\ell=\vect\xi^{(\ell)}\vect\Delta_\mu^{(\ell)}\e<0$, and
$\vect Q_\ell$ is any solution of
$\widetilde{\vect T}^{(\ell)}\vect Q_\ell=-\I_p-\vect\Delta_\mu^{(\ell)}\vect P_\ell$.
The expression in \eqref{eq:occupation-closed-form} is independent of the chosen solution, and every entry of $\vect m_\ell(y)$ grows at most linearly in $y$.
\end{lemma}

{
Lemma~\ref{lem:green} also gives an explicit matrix-transform representation of the launch-density kernel in \eqref{eq:launch-op-correct}. For $x>0$ and a matrix argument $\vect A$, define
\begin{equation}
\label{eq:truncated-transforms}
\begin{aligned}
g_{ij}^{(c,\ell)}(x;\vect A)
&:=\int_{(0,x]} e^{\vect A y}F_{ij}^{(c,\ell)}(\dd y),
&
h_{ij}^{(c,\ell)}(x;\vect A)
&:=\mathscr M_{ij}^{(c,\ell)}(\vect A)-g_{ij}^{(c,\ell)}(x;\vect A).
\end{aligned}
\end{equation}
The truncated transform is well defined for every matrix $\vect A$. In the formula below, the untruncated remainder is evaluated only at $\vect A=\vect U_\ell$, where \eqref{eq:matrix-transform} is finite for every launch law.
}

\begin{corollary}
\label{cor:launch-fc}
{Suppose that the hypotheses of Lemma~\ref{lem:green} hold.} Then, for every $x>0$,
\begin{equation}
\label{eq:launch-fc}
\begin{aligned}
\Lop{c}{\ell}(x)_{i\bullet}
={}&\sum_{j\in E}T^{(c,\ell)}_{ij}\,\bm e_{j}^{\top}
\Big\{
\big[g_{ij}^{(c,\ell)}(x;\vect V_\ell)-g_{ij}^{(c,\ell)}(x;\vect U_\ell)\big]e^{-\vect V_\ell x}
\\
&\hspace{4.4cm}
+h_{ij}^{(c,\ell)}(x;\vect U_\ell)
\big[e^{-\vect U_\ell x}-e^{-\vect V_\ell x}\big]
\Big\}\vect W_\ell,
\qquad i\in E.
\end{aligned}
\end{equation}
\end{corollary}

{Corollary~\ref{cor:launch-fc} is used for the concrete evaluations in Section~\ref{sec:jumps};} its proof is given in Appendix~\ref{app:layer-derivations}. The remaining auxiliary proofs and the detailed product-form argument are also collected there.

\subsection{Product-form stationary distribution}

{We now give the multivariate stationary law under the hold-and-jump and regulated-base conventions. A key ingredient is the family of integrated transfer matrices}
\begin{equation}
\label{eq:A-def-correct}
\vect{M}^{(c,\ell)}=\int_0^\infty \Lop{c}{\ell}(x)\,\dd x,\qquad  0\le c<\ell\le C,
\end{equation}
which appear in the normalization of Theorems~\ref{thm:main} and~\ref{thm:main-reg}.

More explicitly, the $i$th row of \eqref{eq:A-def-correct} is
\begin{align*}
\big(\vect M^{(c,\ell)}\big)_{i\bullet}
=\sum_{j\in E}T^{(c,\ell)}_{ij}\int_{(0,\infty)}\bm e_j^\top\vect m_\ell(y)\,F_{ij}^{(c,\ell)}(\dd y),
\end{align*}
where $\vect m_\ell$ is given by Lemma~\ref{lem:occupation}. {Since its entries grow at most linearly, Assumption~\ref{assump:first-moment} implies that every entry of $\vect M^{(c,\ell)}$ is finite.}

We can now state precisely the product-form densities previewed in the introduction to this section. The stationary distribution is not simply a density spread over all of $\R_+^C$. Whenever fewer than $C$ colors have ever been launched, the unlaunched colors sit exactly at zero rather than near zero, so the occupied colors trace out a lower-dimensional face of $\R_+^C$ instead of filling out the whole space. Such a face has zero $C$-dimensional Lebesgue measure, yet the stationary law can still place positive probability on it. In the hold-and-jump model the most extreme case is when no color has ever been launched, giving an atom at the empty state. We therefore describe the density face by face.

For any nonempty set of occupied colors, write its elements in increasing order as
\begin{align*}
1\le c_1<\cdots<c_n\le C.
\end{align*}
The corresponding face of $\R_+^C$, together with the vector of its positive coordinates, is defined jointly by
\begin{align*}
\Theta_{c_1,\dots,c_n}
&=\set{\bm x\in\R_+^C\st x_{c_1},\dots,x_{c_n}>0 \text{ and } x_c=0 \text{ for } c\notin\{c_1,\dots,c_n\}},\\
\bm x_{>0}&=(x_{c_1},\dots,x_{c_n})^\top\in(0,\infty)^n,\qquad \bm x\in\Theta_{c_1,\dots,c_n}.
\end{align*}
Thus $c_1,\dots,c_n$ are simply the ordered indices of the positive coordinates on this face, and $\bm x_{>0}$ collects exactly those coordinates.

{
For either boundary convention $\mathrm K\in\{\mathrm H,\mathrm R\}$, the stationary density row vector on $\Theta_{c_1,\dots,c_n}$ is denoted $\vect\pi^{\mathrm K}_{c_1,\dots,c_n}(\bm x_{>0})$ and is defined, componentwise, by
\begin{align*}
\bigl(\vect\pi^{\mathrm K}_{c_1,\dots,c_n}(\bm x_{>0})\bigr)_i
\,\dd x_{c_1}\cdots\dd x_{c_n}
&=\lim_{t\to\infty}\Prob\Bigl(X_{c_r}(t)\in\dd x_{c_r},\ r=1,\dots,n,\\
&\hspace{3cm}
X_c(t)=0 \text{ for } c\notin\{c_1,\dots,c_n\},\ J(t)=i\Bigr),
\end{align*}
for $i\in E$ and $\bm x\in\Theta_{c_1,\dots,c_n}$, with $(\bm X(t),J(t))$ on the right-hand side taken under the hold-and-jump process when $\mathrm K=\mathrm H$ and under the regulated-base process when $\mathrm K=\mathrm R$. When the face is clear from $\bm x$, we write this density simply as $\vect\pi^{\mathrm K}(\bm x)$.

In the hold-and-jump model the stationary law also has an atom at the empty state, whose phase mass row vector is denoted $\vect p_0^{\mathrm H}$, with
\begin{align*}
(\vect p_0^{\mathrm H})_i=\lim_{t\to\infty}\Prob\bigl(\bm X(t)=\bm 0,\ J(t)=i\bigr),\qquad i\in E.
\end{align*}
}

In the regulated-base model, only faces containing color $1$ are admissible. {Unlike a coordinate that has depleted and thereafter remains zero while another color is active, the regulated coordinate $X_1$ only touches zero instantaneously while it is active,} an event of probability zero in stationarity, so it is strictly positive at almost every time. Every occupied face in the regulated-base model therefore has $c_1=1$, so their occupied indices satisfy
\begin{align*}
1=c_1<c_2<\cdots<c_n\le C.
\end{align*}

{Throughout the following product formulas, matrix products are taken from left to right in increasing index $r$, and an empty product equals $\I_p$.}

\begin{theorem}[Hold-and-jump product-form representation]
\label{thm:main}
{Suppose that Assumptions~\ref{assump:intensity}, \ref{assump:stability}, and~\ref{assump:first-moment} hold. Then the hold-and-jump model admits a stationary probability law.} On the face $\Theta_{c_1,\dots,c_n}$, the stationary density row vector
of the hold-and-jump model is
\begin{equation}
\label{eq:main-density}
\vect{\pi}^{\mathrm H}(\bm{x})
=
\vect{p}^{\mathrm H}_0
\Lop{0}{c_1}(x_{c_1})
\prod_{r=1}^{n-1}
\Lop{c_r}{c_{r+1}}(x_{c_{r+1}}).
\end{equation}
The boundary mass vector $\vect{p}^{\mathrm H}_0$ satisfies
\begin{equation}
\label{eq:main-boundary}
\vect{p}^{\mathrm H}_0\widetilde{\vect T}^{(0)}
=
\vect{p}^{\mathrm H}_0
\left(
\vect{T}^{(0)}
+
\sum_{c=1}^C\Rop{0}{c}
\right)
=
\vect{0}.
\end{equation}
Moreover, with
\begin{align*}
\mathfrak{M}^{\mathrm H}
=
\begin{bmatrix}
\vect{0} & \vect{M}^{(1,2)} & \cdots & \vect{M}^{(1,C)}\\
\vect{0} & \vect{0} & \cdots & \vect{M}^{(2,C)}\\
\vdots & \ddots & \ddots & \vdots\\
\vect{0} & \cdots & \vect{0} & \vect{0}
\end{bmatrix},
\end{align*}
$\vect{m}^{\mathrm H}=[\vect{M}^{(0,1)},\dots,\vect{M}^{(0,C)}]$, $\e^{\mathrm H}$
the column vector obtained by stacking $C$ copies of $\e$, and $\I^{\mathrm H}$
the identity matrix of the same dimension as $\mathfrak M^{\mathrm H}$, the normalization
condition is
\begin{equation}
\label{eq:normalization}
\vect{p}^{\mathrm H}_0\e
+
\vect{p}^{\mathrm H}_0\vect{m}^{\mathrm H}
(\I^{\mathrm H}-\mathfrak{M}^{\mathrm H})^{-1}
\e^{\mathrm H}
=
1.
\end{equation}
\end{theorem}

\begin{theorem}[Regulated-base product-form representation]
\label{thm:main-reg}
{Suppose that Assumptions~\ref{assump:intensity}, \ref{assump:stability}, and~\ref{assump:first-moment} hold for the regulated-base model. Then the model admits a stationary probability law.} On the face
$\Theta_{1,c_2,\dots,c_n}$, the stationary density row vector is
\begin{equation}
\label{eq:main-density-reg}
\vect{\pi}^{\mathrm R}(\bm{x})
=
\vect{\nu}^{\mathrm R}(x_1)
\prod_{r=1}^{n-1}
\Lop{c_r}{c_{r+1}}(x_{c_{r+1}}),
\qquad x_1>0,
\end{equation}
where $c_1=1$ and $\vect{\nu}^{\mathrm R}$ has the form
\eqref{eq:regulated-base-density}, with the direction of
$\vect g_{\mathrm{reg}}$ fixed by the reflecting boundary condition
\eqref{eq:regulated-base-boundary-g}. Because color $1$ is only one component
of the full stack, $\vect{\nu}^{\mathrm{R}}$ is not normalized to integrate
to $1$ on its own. Its scale is fixed jointly with the excursion mass by the
global normalization below. {Thus, $\vect\nu^{\mathrm R}$ is the density on the base face and an anchor for the full stack, not the marginal stationary density of $X_1$ after higher-color faces have been included.}

With
\begin{align*}
\mathfrak{M}^{\mathrm R}
=
\begin{bmatrix}
\vect{0} & \vect{M}^{(2,3)} & \cdots & \vect{M}^{(2,C)}\\
\vect{0} & \vect{0} & \cdots & \vect{M}^{(3,C)}\\
\vdots & \ddots & \ddots & \vdots\\
\vect{0} & \cdots & \vect{0} & \vect{0}
\end{bmatrix},
\end{align*}
$\vect{m}^{\mathrm R}
=
[\vect{M}^{(1,2)},\dots,\vect{M}^{(1,C)}]$,
$\e^{\mathrm R}$ the column vector obtained by stacking $C-1$
copies of $\e$, and $\I^{\mathrm R}$ the identity matrix of the same dimension
as $\mathfrak M^{\mathrm R}$, the
normalization condition is
\begin{equation}
\label{eq:normalization-reg}
\int_0^\infty
\vect{\nu}^{\mathrm R}(x)\,\dd x
\left[
\e
+
\vect{m}^{\mathrm R}
(\I^{\mathrm R}-\mathfrak{M}^{\mathrm R})^{-1}
\e^{\mathrm R}
\right]
=
1.
\end{equation}
\end{theorem}

\begin{proof}[Proof sketch of Theorems~\ref{thm:main} and~\ref{thm:main-reg}]
{The proof constructs a finite invariant measure directly.} Active-time censoring collapses complete higher-color excursions into the effective intensity matrices \(\widetilde{\vect T}^{(c)}\), so the killed Green kernel of a depleting color gives the real-time occupation density of that color. A second-kind launch from color \(c\) into \(\ell\) therefore multiplies the current sector density by the launch-density kernel \(\Lop{c}{\ell}\); iterating this single-launch balance along \(c_1<\cdots<c_n\) gives the product forms. The hold-and-jump anchor is the empty-state vector \(\vect{p}^{\mathrm H}_0\), while the regulated-base anchor is \(\vect{\nu}^{\mathrm{R}}(x_1)\). Integrating the sector densities and using \(\int_0^\infty\Lop{c}{\ell}(x)\,\dd x=\vect M^{(c,\ell)}\) gives the two normalizations. The detailed stationary launch-balance verification, including the normalization argument, is given in Appendix~\ref{app:layer-derivations}.
\end{proof}

Together with the empty-state atom $\vect{p}^{\mathrm H}_0$ in the hold-and-jump model, these face densities, one for each admissible sequence of occupied colors, account for every part of $\R_+^C$ carrying positive stationary probability, giving the complete stationary distribution. If the joint process is irreducible, this invariant probability law is unique and the process is positive recurrent.

\section{Computational Considerations and Examples}
\label{sec:jumps}

{Every step of the construction above is a finite matrix computation. This section records that recursion and then evaluates the launch-height transforms for four representative laws.}

{First, the matrix integral in \eqref{eq:matrix-transform}, needed to determine} $\overline{\vect\Phi}_{ij}^{(c,\ell)}=\bm e_j^\top\mathscr M_{ij}^{(c,\ell)}(\vect U_\ell)$, can be computed efficiently. In particular, if $\vect U_\ell$ is diagonalizable, i.e. we can write
\[{\vect U_\ell=\vect Z_\ell\diag(\lambda_{\ell,1},\ldots,\lambda_{\ell,p})\vect Z^{-1}_\ell,}\]
where {$\vect Z_\ell$} is the matrix whose columns are eigenvectors of $\vect U_\ell$ and $\lambda_{\ell,1},\ldots,\lambda_{\ell,p}$ the corresponding eigenvalues, then
\begin{align*}
\mathscr M_{ij}^{(c,\ell)}(\vect U_\ell)
={\vect Z_\ell}\diag\!\big(\mathscr M_{ij}^{(c,\ell)}(\lambda_{\ell,1}),\ldots,
\mathscr M_{ij}^{(c,\ell)}(\lambda_{\ell,p})\big){\vect Z^{-1}_\ell}.
\end{align*}
{When $\vect U_\ell$ is not diagonalizable, a Jordan decomposition may be used instead. Under the exponential-moment condition of Remark~\ref{rem:functional-calculus}, the same evaluation admits the contour representation \eqref{eq:matrix-lt}. See \citet{bladt2016use} for details.}

{The recursion runs as follows. The depletion equations are solved from color $C$ downward, the matrix transforms are evaluated and assembled into the return operators, the launch-density and integrated transfer matrices are formed, and the normalization is carried out through the strictly upper-triangular block matrix $\mathfrak M^{\mathrm H}$ or $\mathfrak M^{\mathrm R}$.}

{The examples below rely on the truncated transforms $g_{ij}^{(c,\ell)}$ and $h_{ij}^{(c,\ell)}$ of \eqref{eq:truncated-transforms} for the occupation calculations.}
{
Here $\vect A\in\{\vect U_\ell,\vect V_\ell\}$ and $F_{ij}^{(c,\ell)}$ is the launch-height distribution for a transition from phase $i$ in color $c$ to phase $j$ in color $\ell$. In each example, $H_{ij}^{(c,\ell)}\sim F_{ij}^{(c,\ell)}$. Substitution into Corollary~\ref{cor:launch-fc} gives the launch-density kernel.
}
\begin{example}[Deterministic launch height]
\label{ex:deterministic}
A degenerate launch height needs no functional calculus at all, since the integral defining $\mathscr M_{ij}^{(c,\ell)}(\vect A)$ collapses to a single evaluation of the exponential at the atom $h_{ij}^{(c,\ell)}$, and the same holds for its truncated version $g_{ij}^{(c,\ell)}(x;\vect A)$. For $F_{ij}^{(c,\ell)}=\delta_{h_{ij}^{(c,\ell)}}$, $h_{ij}^{(c,\ell)}>0$,
\begin{align*}
\mathscr M_{ij}^{(c,\ell)}(s)&=e^{s h_{ij}^{(c,\ell)}}, &
\mathscr M_{ij}^{(c,\ell)}(\vect U_\ell)&=e^{\vect U_\ell h_{ij}^{(c,\ell)}},\\
\overline{\vect\Phi}_{ij}^{(c,\ell)}&=\bm e_j^\top e^{\vect U_\ell h_{ij}^{(c,\ell)}}, &
g_{ij}^{(c,\ell)}(x;\vect A)&=\1_{\{h_{ij}^{(c,\ell)}\le x\}}e^{\vect A h_{ij}^{(c,\ell)}}.
\end{align*}
The complementary transform uses the indicator $\1_{\{h_{ij}^{(c,\ell)}>x\}}$, and the phase pair contributes $T_{ij}^{(c,\ell)}\bm e_j^\top\vect G_\ell(h_{ij}^{(c,\ell)},x)$ to $\Lop{c}{\ell}(x)$.
\end{example}
\begin{example}[Phase-type launch height]
\label{ex:ph}
A phase-type launch height has a rational Laplace transform in closed form, and its matrix version follows the same functional-calculus substitution as \eqref{eq:matrix-lt}, except that the operator being inverted now has to act jointly on the MMBM's own phase space and on the auxiliary phase-type states. This joint action is expressed through a Kronecker sum, using the fact that the exponential of a Kronecker sum factors as a Kronecker product of exponentials, $e^{(\vect A\otimes\I_q+\I_p\otimes\vect S)y}=e^{\vect Ay}\otimes e^{\vect Sy}$. Throughout, $\otimes$ denotes the Kronecker product of matrices. Let $H_{ij}^{(c,\ell)}\sim\mathrm{PH}(\vect\alpha_{ij}^{(c,\ell)},\vect S_{ij}^{(c,\ell)})$, with $\vect s_{ij}^{(c,\ell)}=-\vect S_{ij}^{(c,\ell)}\e$, and define
\begin{align*}
\vect{\mathcal K}_{ij}^{(c,\ell)}(\vect A)=\vect A\otimes\I_q+\I_p\otimes\vect S_{ij}^{(c,\ell)}.
\end{align*}
Then
\begin{align*}
\mathscr M_{ij}^{(c,\ell)}(s)
&=\vect\alpha_{ij}^{(c,\ell)}[-(\vect S_{ij}^{(c,\ell)}+s\I_q)^{-1}]\vect s_{ij}^{(c,\ell)},\\
\mathscr M_{ij}^{(c,\ell)}(\vect U_\ell)
&=(\I_p\otimes\vect\alpha_{ij}^{(c,\ell)})[-\vect{\mathcal K}_{ij}^{(c,\ell)}(\vect U_\ell)^{-1}](\I_p\otimes\vect s_{ij}^{(c,\ell)}),\\
g_{ij}^{(c,\ell)}(x;\vect A)
&=(\I_p\otimes\vect\alpha_{ij}^{(c,\ell)})\big[x\phi_1(x\vect{\mathcal K}_{ij}^{(c,\ell)}(\vect A))\big](\I_p\otimes\vect s_{ij}^{(c,\ell)}),
\end{align*}
where $\phi_1(\vect Z)=\sum_{n\ge0}\vect Z^n/(n+1)!$. Erlang laws are included as a special case.
\end{example}
\begin{example}[Weibull launch height]
\label{ex:weibull}
{Let $H_{ij}^{(c,\ell)}\sim\mathrm{Weibull}(k_{ij}^{(c,\ell)},\lambda_{ij}^{(c,\ell)})$. For $k_{ij}^{(c,\ell)}>1$ the moment generating function is entire, so the moment series converges on the whole complex plane and may be evaluated term by term at $\vect U_\ell$,}
\begin{align*}
\mathscr M_{ij}^{(c,\ell)}(s)
&=\sum_{n=0}^{\infty}\frac{\Gamma(1+n/k_{ij}^{(c,\ell)})}{n!}(\lambda_{ij}^{(c,\ell)}s)^n,\\
\mathscr M_{ij}^{(c,\ell)}(\vect U_\ell)
&=\sum_{n=0}^{\infty}\frac{\Gamma(1+n/k_{ij}^{(c,\ell)})}{n!}(\lambda_{ij}^{(c,\ell)}\vect U_\ell)^n.
\end{align*}
For $k_{ij}^{(c,\ell)}=1$ the law is exponential and
\begin{align*}
\mathscr M_{ij}^{(c,\ell)}(\vect U_\ell)=(\I_p-\lambda_{ij}^{(c,\ell)}\vect U_\ell)^{-1}.
\end{align*}
{When $0<k_{ij}^{(c,\ell)}<1$ the law still satisfies Assumption~\ref{assump:first-moment} and remains admissible for the stationary theory, but it has no positive exponential moment and the displayed moment series has zero radius of convergence. In that case $\mathscr M_{ij}^{(c,\ell)}(\vect U_\ell)$ is evaluated from the defining matrix integral \eqref{eq:matrix-transform}, or from the spectral formula above, where the scalar transform equals $1$ at the zero eigenvalue and is the Laplace transform at the remaining ones. Truncated transforms follow from \eqref{eq:truncated-transforms}.}
\end{example}
\begin{example}[Gamma launch height]
\label{ex:gamma}
The Gamma law again has a closed-form Laplace transform, so its matrix version follows directly from \eqref{eq:matrix-lt} exactly as in the deterministic and phase-type cases, with the fractional power $(\beta_{ij}^{(c,\ell)}\I_p-\vect U_\ell)^{-\alpha_{ij}^{(c,\ell)}}$ defined through the same holomorphic branch used throughout. Let $H_{ij}^{(c,\ell)}\sim\Gamma(\alpha_{ij}^{(c,\ell)},\beta_{ij}^{(c,\ell)})$ with rate parameterization. Then
\begin{align*}
\mathscr M_{ij}^{(c,\ell)}(s)
&=\left(\frac{\beta_{ij}^{(c,\ell)}}{\beta_{ij}^{(c,\ell)}-s}\right)^{\alpha_{ij}^{(c,\ell)}},\\
\mathscr M_{ij}^{(c,\ell)}(\vect U_\ell)
&=(\beta_{ij}^{(c,\ell)})^{\alpha_{ij}^{(c,\ell)}}
(\beta_{ij}^{(c,\ell)}\I_p-\vect U_\ell)^{-\alpha_{ij}^{(c,\ell)}}.
\end{align*}
The truncated transforms are the corresponding lower and upper incomplete-gamma matrix functions. Integer shape gives the Erlang phase-type case.
\end{example}

These examples show that the framework accommodates both deterministic and random launch heights, discrete random laws included, not only the continuous ones written out above. Once the corresponding scalar transform is specified, the return probabilities and launch-density kernels follow from the same matrix-functional construction, separating the workload-stack dynamics from the particular choice of launch-height distribution.

\section{Concluding Remarks}
\label{sec:conclusion}

We introduced a colored Markov-modulated Brownian motion model for nested preemption, in which a Brownian active layer sits on top of a stack of suspended work, and is itself preempted whenever new work arrives. Coloring not only identifies both the stack position and the priority of each layer, but also serves as a guarantee that the suspended stack can be described as a finite backward recursion. Specifically, working from the top color down, each higher-color excursion is folded into the dynamics of the color below it through a return matrix, which summarizes the excursion's net effect on the phase process. The excursion's contribution to the stationary density is captured separately by a launch-density kernel, which records the occupation time spent at each height before depletion.

Strictly positive random launch heights serve a double purpose in this construction. They represent the amount of work carried by a newly arriving or generated item, and they also remove the boundary ambiguity that would otherwise arise from a Brownian layer launched exactly at zero, whose path oscillates on both sides of the boundary immediately after being born there. The hold-and-jump and regulated-base boundary conventions share this entire higher-color machinery, and differ only in how the bottom of the stack behaves once every layer above it has depleted. The former holds at an empty state until new work arrives, contributing an atom to the stationary law. The latter reflects the lowest color at zero, contributing a continuous density anchored at the base.

Both conventions lead to an explicit product-form stationary density with one factor for each occupied color. Each factor is built from matrix exponentials and the launch-density kernels. Although the normalization sums over every admissible sequence of occupied colors, it reduces to a finite matrix computation. Under the finite-mean assumption, the return and launch operators are defined through matrix integrals for arbitrary launch-height laws, heavy-tailed ones included. When the launch height has a positive exponential moment, Remark~\ref{rem:functional-calculus} also gives the contour representation. The deterministic, phase-type, Weibull, and gamma examples of Section~\ref{sec:jumps} cover both cases. Practitioners may therefore change the launch-height distribution without repeating the underlying matrix-analytic derivation.

We identify two natural directions for further development. The first concerns interacting stacks, in which several colored workload stacks share resources or trigger launches in one another. This framework falls under the broader problem of multivariate stochastic fluid processes or multivariate MMBMs, which has remained difficult to handle in the literature, and would need progress on its own before its application to colored MMBM could be carried through.

The second extension concerns controlled launch or priority mechanisms. In such models, the color ordering or the launch-height distribution associated with a transition may depend on the current state of the stack, leading to state-dependent launch dynamics. This route seems more tractable, though the resulting process would likely fall outside the usual matrix-exponential framework, as inhomogeneity in the launch height breaks the time-homogeneity the current construction relies on.

\appendix
\section{Proofs of the Product-Form Representations}
\label{app:layer-derivations}

{This appendix proves the closed form for the Green kernel $\vect G_\ell(y,x)$, the integrated occupation kernel $\vect m_\ell(y)$, the launch-density formula, and the product-form Theorems~\ref{thm:main} and~\ref{thm:main-reg}. The active-time censoring result used throughout is stated and proved in Lemma~\ref{lem:active-censoring}.}

\begin{proof}[Proof of Lemma~\ref{lem:green}]
Let $\mathcal W_\ell$ be the scale matrix of the color-$\ell$ MMBM, with $\mathcal W_\ell(z)=\vect0$ for $z<0$. The explicit MMBM scale matrix on p.~555 of \citet{Breuer2012Occupation}, written in terms of the two roots $\vect U_\ell$ and $\vect V_\ell$, is
\begin{align*}
\mathcal W_\ell(z)=\big(e^{\vect V_\ell z}-e^{\vect U_\ell z}\big)\vect W_\ell,
\qquad z\geq0.
\end{align*}
Let $\vect R_\ell^{\rm pot}=-\vect W_\ell^{-1}\vect V_\ell\vect W_\ell$, so that
$e^{\vect R_\ell^{\rm pot}x}=\vect W_\ell^{-1}e^{-\vect V_\ell x}\vect W_\ell$.
The terminating-barrier potential-density identity in
\citet[Theorem~1, Equation~(13), together with Section~8]{Ivanovs2014}, after shifting the initial level from $y$ to $0$, gives
\begin{align*}
\vect G_\ell(y,x)=\mathcal W_\ell(y)e^{\vect R_\ell^{\rm pot}x}-\mathcal W_\ell(y-x),
\qquad x>0.
\end{align*}
For $0\leq y\leq x$, the second term vanishes, which gives the first branch of \eqref{eq:green-closed-form}. For $y\geq x$,
\begin{align*}
\vect G_\ell(y,x)
&=\big(e^{\vect V_\ell y}-e^{\vect U_\ell y}\big)e^{-\vect V_\ell x}\vect W_\ell
-\big(e^{\vect V_\ell(y-x)}-e^{\vect U_\ell(y-x)}\big)\vect W_\ell
\\
&=\big(e^{\vect U_\ell(y-x)}-e^{\vect U_\ell y}e^{-\vect V_\ell x}\big)\vect W_\ell,
\end{align*}
which is the second branch.
\end{proof}

\begin{proof}[Proof of Lemma~\ref{lem:occupation}]
By \eqref{eq:green-def} and Tonelli's theorem, $[\vect m_\ell(y)]_{ij}$ is the expected {active time} spent in phase $j$ before depletion when the process starts from active color $\ell$, level $y$, and phase $i$. Integrating the two branches of \eqref{eq:green-closed-form} over $x$ and using the invertibility of $\vect V_\ell$ gives
\begin{align}
\label{eq:occupation-integrated-green}
\vect m_\ell(y)
=\int_0^y e^{\vect U_\ell s}\vect W_\ell\,\dd s
+(\I_p-e^{\vect U_\ell y})\vect V_\ell^{-1}\vect W_\ell.
\end{align}
Put $\vect A_\ell=\tfrac12(\vect\Delta_\sigma^{(\ell)})^2$, and let $\vect\beta_\ell$ be the left null vector of $\vect U_\ell$ normalized by $\vect\beta_\ell\e=1$. Multiplying \eqref{eq:layer-qme} on the left by $\vect\xi^{(\ell)}$ shows that
$\vect\xi^{(\ell)}(\vect A_\ell\vect U_\ell+\vect\Delta_\mu^{(\ell)})$ is a left null vector of $\vect U_\ell$; hence
\begin{align*}
\vect\beta_\ell
=d_\ell^{-1}\vect\xi^{(\ell)}
(\vect A_\ell\vect U_\ell+\vect\Delta_\mu^{(\ell)}).
\end{align*}
Since $\vect V_\ell$ is invertible, its quadratic equation and
$\vect\xi^{(\ell)}\widetilde{\vect T}^{(\ell)}=\vect0$ give
$\vect\xi^{(\ell)}(\vect A_\ell\vect V_\ell+\vect\Delta_\mu^{(\ell)})=\vect0$.
Together with $\vect A_\ell(\vect V_\ell-\vect U_\ell)\vect W_\ell=\I_p$, this yields
\begin{align*}
\vect\beta_\ell\vect W_\ell
=-d_\ell^{-1}\vect\xi^{(\ell)}
=|d_\ell|^{-1}\vect\xi^{(\ell)}.
\end{align*}
Thus the spectral projection of $\vect U_\ell$ at its simple zero eigenvalue maps $\vect W_\ell$ to
$\vect P_\ell=|d_\ell|^{-1}\e\vect\xi^{(\ell)}$. Since all remaining eigenvalues of $\vect U_\ell$ lie in the open left half-plane, the matrix
\begin{align*}
\vect S_\ell
:=\int_0^\infty\big(e^{\vect U_\ell s}\vect W_\ell-\vect P_\ell\big)\,\dd s
\end{align*}
is well defined. Moreover, $e^{\vect U_\ell y}\vect P_\ell=\vect P_\ell$, and therefore
\begin{align*}
\int_0^y e^{\vect U_\ell s}\vect W_\ell\,\dd s
=y\vect P_\ell+(\I_p-e^{\vect U_\ell y})\vect S_\ell.
\end{align*}
Defining $\vect Q_\ell=\vect S_\ell+\vect V_\ell^{-1}\vect W_\ell$ and substituting into \eqref{eq:occupation-integrated-green} gives \eqref{eq:occupation-closed-form}.

It remains to verify the stated characterization of $\vect Q_\ell$. From its definition,
$\vect U_\ell\vect S_\ell=\vect P_\ell-\vect W_\ell$. Using the two quadratic equations and
$\vect A_\ell(\vect V_\ell-\vect U_\ell)\vect W_\ell=\I_p$, we obtain
\begin{align*}
\widetilde{\vect T}^{(\ell)}\vect Q_\ell
&=\widetilde{\vect T}^{(\ell)}\vect S_\ell
+\widetilde{\vect T}^{(\ell)}\vect V_\ell^{-1}\vect W_\ell
\\
&=\vect A_\ell\vect U_\ell\vect W_\ell
-\vect\Delta_\mu^{(\ell)}\vect P_\ell
+\vect\Delta_\mu^{(\ell)}\vect W_\ell
-\vect A_\ell\vect V_\ell\vect W_\ell
-\vect\Delta_\mu^{(\ell)}\vect W_\ell
\\
&=-\I_p-\vect\Delta_\mu^{(\ell)}\vect P_\ell.
\end{align*}
{The linear system is solvable because $\vect\xi^{(\ell)}(-\I_p-\vect\Delta_\mu^{(\ell)}\vect P_\ell)=\vect0$.} Any other solution differs by $\e\vect a$ for some row vector $\vect a$, and the expression in \eqref{eq:occupation-closed-form} is unchanged because
$(\I_p-e^{\vect U_\ell y})\e=\vect0$. The same formula also shows that every entry of $\vect m_\ell(y)$ grows at most linearly.
\end{proof}

\begin{proof}[Proof of Corollary~\ref{cor:launch-fc}]
Split the integral defining $\Lop{c}{\ell}(x)_{i\bullet}$ in \eqref{eq:launch-op-correct} at $y=x$, and substitute the two branches of \eqref{eq:green-closed-form} on each piece. On $(0,x]$,
\[\int_0^x\bm e_j^\top(e^{\vect V_\ell y}-e^{\vect U_\ell y})F_{ij}^{(c,\ell)}(\dd y)\,e^{-\vect V_\ell x}\vect W_\ell=\bm e_j^\top[g_{ij}^{(c,\ell)}(x;\vect V_\ell)-g_{ij}^{(c,\ell)}(x;\vect U_\ell)]e^{-\vect V_\ell x}\vect W_\ell.\]
On $(x,\infty)$, write $e^{\vect U_\ell(y-x)}=e^{\vect U_\ell y}e^{-\vect U_\ell x}$, so that
\[\int_x^\infty\bm e_j^\top e^{\vect U_\ell y}F_{ij}^{(c,\ell)}(\dd y)=\bm e_j^\top h_{ij}^{(c,\ell)}(x;\vect U_\ell).\]
This gives the second bracket, since $e^{-\vect U_\ell x}$ commutes with $h_{ij}^{(c,\ell)}(x;\vect U_\ell)$, both being functions of $\vect U_\ell$. Adding the two branches, multiplying by $T^{(c,\ell)}_{ij}$, and summing over $j$ gives \eqref{eq:launch-fc}.
\end{proof}

\begin{proof}[Detailed proof of Theorems~\ref{thm:main} and~\ref{thm:main-reg}]
{A finite invariant measure is constructed. In the hold-and-jump model the construction starts from an arbitrary positive multiple of the invariant row vector of the censored empty-state generator $\widetilde{\vect T}^{(0)}$, and in the regulated-base model from an arbitrary positive multiple of the base-face anchor $\vect\nu^{\mathrm R}(x_1)\,\dd x_1$. Censoring supplies the occupation kernel for a single depleting color, a launch maps a source-face measure into the occupation measure of the newly active color, and iterating over increasing color sequences produces every face measure. Finite first moments make the total mass finite, and a single global normalization then yields an invariant probability law.}

Fix a depleting color $\ell$. {By Lemma~\ref{lem:active-censoring}, the phase process observed while color $\ell$ is active has intensity matrix $\widetilde{\vect T}^{(\ell)}$, and its active-time occupation equals the real-time occupation restricted to $\{A(t)=\ell\}$. The Green kernel $\vect G_\ell(y,x)$ therefore gives the expected real-time occupation density on the face where color $\ell$ is active, started at height $y$.}

{
Fix colors $c<\ell$ and let $\vect\eta_c(\bm z,y)$ be a candidate invariant source-face density over phase, with lower coordinates $\bm z$ frozen and active height $y$ in color $c$. Since $\vect T^{(c,\ell)}$ is independent of the continuous coordinates, the outgoing launch flux from phase $i$ to phase $j$ is
$\eta_{c,i}(\bm z,y)T_{ij}^{(c,\ell)}\,\dd y$.

At such a launch, an independent height $Y\sim F_{ij}^{(c,\ell)}$ is sampled. Lemma~\ref{lem:active-censoring} and the Green kernel give the occupation generated on the target face as $\bm e_j^\top\vect G_\ell(Y,x_\ell)$. This occupation does not depend on the frozen height $y$, because both the launch law and the subsequent higher-color dynamics depend on the source state only through the triggering phase. Averaging over $Y$ and summing over $(i,j)$ therefore maps the source density to $\vect\eta_c(\bm z,y)\Lop{c}{\ell}(x_\ell)$. This is the unnormalized invariant occupation flow from the source face to the face on which color $\ell$ is active.

Equivalently, the single-launch balance identity is
}
\begin{align*}
\vect\eta_{c\to\ell}(\bm z,y,x_\ell)
=\vect\eta_c(\bm z,y)\,\Lop{c}{\ell}(x_\ell)
=\sum_{i,j\in E}\eta_{c,i}(\bm z,y)T^{(c,\ell)}_{ij}
\int_{(0,\infty)}\bm e_j^\top\vect G_\ell(w,x_\ell)F_{ij}^{(c,\ell)}(\dd w),
\end{align*}
{Here $\vect\eta_{c\to\ell}(\bm z,y,x_\ell)$ is the density of the state right after color $c$ freezes at $y$ and color $\ell$ becomes active at $x_\ell$, with the launch height $w$ already integrated out inside $\mathcal L$. The coordinate $y$ passes through unchanged on both sides, since neither the transition rate nor the launch law depends on it. The source density is multiplied by the launch-rate matrix, and the killed occupation kernel is integrated over the launch law. The same mechanism applies at every further launch along the chain.}

At each step along that chain, the launch height and the excursion that follows depend on the past only through the triggering phase, and the single-launch balance above already integrates the frozen height out of the kernel. Applying it repeatedly along $c_1<c_2<\cdots<c_n$ builds up the unnormalized stationary occupation density $\vect\rho\,\Lop{c_0}{c_1}(x_{c_1})\prod_{r=1}^{n-1}\Lop{c_r}{c_{r+1}}(x_{c_{r+1}})$, where the anchor $\vect\rho$ and the starting index $c_0\in\{0,1\}$ depend on the model, identified next.

In the hold-and-jump model, the bottom of the stack is the empty state, with stationary vector $\vect{p}^{\mathrm H}_0$ solving the balance equation \eqref{eq:main-boundary}. Taking $\vect\rho=\vect{p}^{\mathrm H}_0$ and $c_0=0$ gives \eqref{eq:main-density}.

In the regulated-base model, color $1$ is never removed. By the censoring argument above, its own intensity matrix $\widetilde{\vect T}^{(1)}$ already folds in every excursion above color $1$, so its density has the regulated-MMBM form $\vect{\nu}^{\mathrm{R}}(x_1)$ of \eqref{eq:regulated-base-density}, with direction fixed by the reflecting boundary condition \eqref{eq:regulated-base-boundary-g}. Since $X_1$ touches zero with probability zero in stationarity, no separate boundary atom is needed at $x_1=0$. Taking $\vect\rho=\vect{\nu}^{\mathrm{R}}(x_1)$ and $c_0=1$, and applying the single-launch and chaining arguments above to launches out of color $1$, gives \eqref{eq:main-density-reg}. The overall scale of $\vect{\nu}^{\mathrm{R}}$ is only fixed once the normalization below accounts for the excursion mass.

{
The single-launch identity is the stationary balance between a source face and the excursion faces it creates. Iterating it from an invariant anchor therefore gives an invariant measure on the full disjoint state space. It remains to prove that this measure is finite and fix its scale. Summing \eqref{eq:main-density} over all admissible color sequences, integrating over each face, and using $\int_0^\infty\Lop{c}{\ell}(x)\,\dd x=\vect M^{(c,\ell)}$ from \eqref{eq:A-def-correct}, the hold-and-jump mass equals
\begin{align*}
\vect p_0^{\mathrm H}\e
+\vect p_0^{\mathrm H}\vect m^{\mathrm H}
(\I^{\mathrm H}-\mathfrak M^{\mathrm H})^{-1}\e^{\mathrm H}.
\end{align*}
Every block of $\mathfrak M^{\mathrm H}$ is finite under Assumption~\ref{assump:first-moment}. Since launches only move to larger colors, $\mathfrak M^{\mathrm H}$ is strictly upper triangular and nilpotent, so its inverse is the finite sum $\sum_{k=0}^{C-1}(\mathfrak M^{\mathrm H})^k$. The total mass is therefore finite and positive, and normalization gives \eqref{eq:normalization} and the stationary probability law in Theorem~\ref{thm:main}. The regulated-base construction is identical after replacing the empty-state anchor by $\int_0^\infty\vect\nu^{\mathrm R}(x)\,\dd x$ and restricting excursion colors to $\ell\ge2$, which yields \eqref{eq:normalization-reg} and Theorem~\ref{thm:main-reg}.
}
\end{proof}

\bibliographystyle{apalike}
\bibliography{references_preemptive_workload_stacks}
\end{document}